\documentclass{article}
\usepackage[margin=1in]{geometry}
\usepackage{amsmath}
\usepackage{amssymb}
\usepackage{amsthm}
\usepackage{graphicx}
\usepackage{subfigure}
\usepackage{cite}
\usepackage{comment}
\usepackage{bm}
\usepackage{color}
\usepackage[colorlinks=true]{hyperref}

\newtheorem{theorem}{Theorem}[section]
\newtheorem{corollary}[theorem]{Corollary}
\newtheorem{lemma}[theorem]{Lemma}
\newtheorem{proposition}[theorem]{Proposition}

\theoremstyle{definition}

\newcommand{\norm}[1]{\left\lVert#1\right\rVert}

\DeclareMathOperator{\csch}{csch}

\newcommand{\TheTitle}{Control Problems with Vanishing Lie Bracket Arising from Complete Odd Circulant Evolutionary Games}

\title{\TheTitle}

\author{Christopher Griffin\footnote{Applied Research Laboratory, Penn State University, University Park, PA 16802, E-mail: griffinch@psu.edu} \and James Fan\footnote{Naval Postgraduate School, Monterey, CA 93943, E-mail: james.fan@nps.edu}
}

\begin{document}

\maketitle

\begin{abstract}

 We study an optimal control problem arising from a generalization of rock-paper-scissors in which the number of strategies may be selected from any positive odd number greater than 1 and in which the payoff to the winner is controlled by a control variable $\gamma$. Using the replicator dynamics as the equations of motion, we show that a quasi-linearization of the problem admits a special optimal control form in which explicit dynamics for the controller can be identified. We show that all optimal controls must satisfy a specific second order differential equation parameterized by the number of strategies in the game. We show that as the number of strategies increases, a limiting case admits a closed form for the open-loop optimal control. In performing our analysis we show necessary conditions on an optimal control problem that allow this analytic approach to function.

\end{abstract}

\section{Introduction}

In this paper, we study the frequently occurring phenomenon of cyclic competition in nature \cite{JB75,SL96,KRFB02,GHS03,KR04,KNS05,NHK11}. Cyclic competition in coral reef populations are studied in \cite{JB75}. Sinervo and Lively first characterized rock-paper-scissors like competition in lizards \cite{SL96}, while Gilg, Hanski and Sittler \cite{GHS03} study this behavior in rodents. Cyclic behavior in microbial populations is studied in \cite{KRFB02,KR04,KNS05,NHK11}. In classical and evolutionary game theory, cyclic dominance (e.g., matching pennies, rock-paper-scissors) games are commonly studied  \cite{Mor94,Wei95}. Biologically speaking, in an idealized cyclic game, the absolute fitness measure (payoff) resulting from species interaction can be represented by a circulant matrix, in which row $k$ is a rotation of row $k-1$ for each $k$. Games with circulant payoff matrices have been studied extensively in evolutionary game theory \cite{Ze80,SSW80,HS98,HS03,DG09,HS00,SMJS14,GK16} and provide some of the most interesting behaviors \cite{SMJS14}.

In early work, cyclic interaction is considered without explicit reference to games. Cyclic (chemo-biological) interactions are studied extensively in \cite{SSW78,HSSW80,SSW79} in which both competitive and cooperative behaviors are identified. Analysis of the replicator in which a circulant matrix emerges is studied in \cite{SSW80} as a result of cyclic mass interaction kinetics. Zeeman \cite{Ze80} made an early study of the dynamics of cyclic games showing that in rock-paper-scissors a degenerate Hopf bifurcation leads to the emergence of a non-linear center with no limit cycle possible in any 3 strategy game under the replicator dynamics.  Since this early work, several authors have investigated various cyclic games and games characterized by circulant matrices. Among many other works: Hofbauer and Schlag \cite{HS00} consider imitation in cyclic games; Diekmann and Gils specifically study the cyclic replicator dynamics and focus on the properties of low-dimensional cyclic games \cite{DG09}; Ermentrout et al. consider a transition matrix evolutionary dynamic in which a limit cycle emerges in the rock-paper-scissors game \cite{EGB16}; and Griffin and Belmonte \cite{GB17} study a triple public goods game and show that is is diffeomorphic to generalized rock-paper-scissors. Each of these works focuses explicitly on classes of circulant games, while recent work by Grani{\'{c}} and Kerns \cite{GK16} characterizes the Nash equilibria of arbitrary circulant games, but does not focus on the evolutionary game context. Rock paper scissors has been studied extensively in the literature starting with \cite{ML75}. \cite{CYHP14} studies RPS at the mesoscopic scale, while chaotic behavior in special forms of RPS are studied in \cite{HZTW19,SAF02,SC03,SAC05}. \cite{NHK11} studies the evolution of restraint in a RPS context. \cite{M10} studies the effect of mutation. RPS is used as an exemplar in a non-standard evolutionary dynamic in\cite{GJW20}.

There has also been extensive work on spatial games with circulant payoff matrices. Peltom\"aki and Alvara \cite{PA08} consider both 3 and 4 state rock-paper-scissors. Other papers consider rock-paper-scissors with variations on reaction rate \cite{HMT10} or study the basins of attraction \cite{SWYL10}. DeForest and Belmonte \cite{dB13} study a fitness gradient variation on the spatial replicator and show rock-paper-scissors can exhibit spatial chaos under these dynamics. More recent work by Szczesny et. al \cite{SMR14} considers spiral formations in rock-paper-scissors. Spatial rock-paper-scissors not motivated by the replicator equation is studied extensively in \cite{PR19,SMR14,SMJS14,RMF07,RMF08,HMT10,PR17,SMR13}.

In this paper, we extend work in \cite{GB17} by studying an optimal control problem defined on a $N$-strategy ($N=3,5,7,\dots$) generalization of rock-paper-scissors. Odd cardinality interactions are interesting because they model specific  biological cases \cite{SL96,GHS03,SSW78,HSSW80,SSW79}. Additionally, when $N$ is very large, these have the potential to model systems in which many individuals with a variety of strengths and weaknesses interact. Recent work by \cite{RK20} studies the replicator as emerging from the minimization of an action. This work is related to but distinct from the work studied in this paper.

For this paper, we assume that the payoff matrix is (i) defined by the sum of two circulant matrices, and (ii) one of those circulant matrices admits a single control parameter. Thus we consider the general class of control problems first studied in a specific case in \cite{GB17}. Our payoff matrix is inspired by the generalized rock-paper-scissors matrix defined in \cite{Wei95}. Since every pair of heterogeneous strategic interactions (e.g., rock vs. scissors) results in a non-zero payoff, we refer to this class of games as \textit{complete odd circulant games}.

The main results in this paper are:
\begin{enumerate}

\item We generalize the control problem defined in \cite{GB17} to complete odd circulant games of any order by writing the replicator dynamics as the sum of an uncontrolled component and a controlled component, both of which have circulant Jacobian matrices.

\item We show that a quasi-linearization of the control problem (as done in  \cite{GB17}) has special form admitting a complete characterization of the dynamics of the optimal control. We also derive a sufficient condition on control optimality and thus completely generalize the results in \cite{GB17} to arbitrary complete odd circulant games.

\item As a part of the generalization, we find a second order ordinary differential equation that the optimal control must obey and show that the asymptotic form (as $N$ grows large) of this ODE has a natural closed form solution.
\end{enumerate}

The remainder of this paper is organized as follows: In Section \ref{sec:Prelim} we present preliminary results and notation. In Section \ref{sec:GeneralControl} we introduce the control problem of interest and study a general class of optimal control problems that will assist in the derivation of our main results. Our main results on control of complete odd circulant games are found in Section \ref{sec:CyclicControl}. Conclusions and future directions are presented in Section \ref{sec:Conclusion}. In addition, we provide three appendices. Appendix \ref{sec:OptimalControl} provides essential results from optimal control theory used in this paper. Appendix \ref{sec:TechProofs} provides technical proofs to four lemmas used in the paper. Appendix \ref{app:Examples} contains numerical examples. 

\section{Notation and Preliminary Results}\label{sec:Prelim}
A \textit{circulant matrix} is a square matrix with form:
\begin{displaymath}
\mathbf{A} =
\begin{bmatrix}
a_0 & a_{n-1} & a_{n-2} & \cdots & a_{1}\\
a_1 & a_0 & a_{n-1} & \cdots & a_{2}\\
\vdots & \vdots & \vdots & \ddots & \vdots\\
a_{n-1} & a_{n-2} & a_{n-3} & \cdots & a_0
\end{bmatrix}.
\end{displaymath}
A circulant matrix is entirely characterized by its first row and all other rows are cyclic permutations of this first row. The set of $N \times N$ circulant matrices forms a commutative algebra, a fact that will be used frequently in this paper. Moreover, the eigenvalues of these matrices have special form. If $\mathbf{A}$ is an $N \times N$ circulant matrix and $\omega_0,\dots,\omega_{N-1}$ are the $N^\text{th}$ roots of unity, then eigenvalue $\lambda_j$ ($j=0,\dots,N-1$) is given by the expression:
\begin{displaymath}
\lambda_j = a_0 + a_{n-1}\omega_j + a_{n-2}\omega_j^2 + \cdots + a_{1}\omega_j^{N-1}.
\end{displaymath}
Further details on this class of matrices is available in \cite{D12}.

Let:
\begin{displaymath}
\Delta_N = \left\{\mathbf{u} \in \mathbb{R}^N : \mathbf{1}^T\mathbf{u} = 1, \mathbf{u} \geq \mathbf{0}\right\}
\end{displaymath}
be the unit $N$-simplex embedded in $N$-dimensional Euclidean space. Here $\mathbf{1}$ is an appropriately sized vector of $1$'s and $\mathbf{0}$ is a zero vector. Circulant matrices are a special subclass of Toeplitz matrices and as such inherit all their properties and more.

We consider a family of control problems defined on parameterized cyclic games with $N = 2n + 1$ strategies, where $n = 1,2,\dots$. 
For the remainder of this paper, define $\mathbf{L}_N,\mathbf{M}_N \in \mathbb{R}^{N \times N}$ circulant matrices so that the first row of $\mathbf{L}_N$ is given by:
\begin{equation}
\mathbf{L}_{1\cdot} = \left[{0,-1,1,-1,1,\dots,-1,1}\right]
\end{equation}
and
\begin{equation}
\mathbf{M}_{1\cdot} = \left[{0,0,1,0,1,\cdots,0,1}\right]
\end{equation}

By way of example, we illustrate the matrices $\mathbf{L}_5$ and $\mathbf{M}_5$ for the 5-strategy cyclic game.
\begin{displaymath}
\mathbf{L}_5 =
\begin{bmatrix}
0 & -1 & 1 & -1 & 1 \\
 1 & 0 & -1 & 1 & -1 \\
 -1 & 1 & 0 & -1 & 1 \\
 1 & -1 & 1 & 0 & -1 \\
 -1 & 1 & -1 & 1 & 0
\end{bmatrix}, \qquad
\mathbf{M}_5 =
\begin{bmatrix}
0 & 0 & 1 & 0 & 1 \\
 1 & 0 & 0 & 1 & 0 \\
 0 & 1 & 0 & 0 & 1 \\
 1 & 0 & 1 & 0 & 0 \\
 0 & 1 & 0 & 1 & 0
\end{bmatrix}.
\end{displaymath}
From a game-theoretic perspective, we can think of $\mathbf{L}_N$ as being the traditional payoff matrix of the complete cyclic game with $N$ strategies. $\mathbf{M}_N$ can be thought of an actuating matrix that will determine whether the interior fixed point of the complete cyclic game is stable or unstable.

In the remainder of this paper, we will consider the generalized cyclic game with $N$ strategies where $N = 3, 5, 7, \dots$, and we note that both $L_N$ and $M_N$ are circulant matrices. The payoff matrix for the generalized cyclic game with $N$ strategies and parameter $\gamma$ is:
\begin{displaymath}
\mathbf{A}_N(\gamma) = \mathbf{L}_N + \gamma\mathbf{M}_N.
\end{displaymath}
If $\gamma = 0$ and $N = 3$, then $\mathbf{A}_3(\gamma)$ is just the rock-paper-scissors matrix. Without loss of generality, we assume $\gamma > -1$. Otherwise, the natural winning precedence in the cyclic game is reversed.


In the control problem defined in the sequel, the replicator dynamics are the nonlinear equations of motion with control parameter $\gamma$:
\begin{equation}
\mathcal{S}_N =\left\{
\begin{aligned}
&\dot{\mathbf{u}}_i = \mathbf{u}_i\left((\mathbf{e}_i - \mathbf{u})^T\mathbf{A}_N(\gamma)\mathbf{u}\right),\\
&\mathbf{u}(0) = \mathbf{u}_0.
\end{aligned}\right.
\label{eqn:RepDyn}
\end{equation}
Here $\mathbf{u} = \langle{u_1,\dots,u_N}\rangle$ is the vector denoting the proportion of the population playing each of the $N$ strategies. It is well known \cite{HS98,Wei95} that if $\mathbf{u}_0 \in \Delta_N$, then $\mathbf{u}(t)$ is confined to $\Delta_N$ for all time. For the remainder of this paper, we assume $\mathbf{u}_0 \in \Delta_N$.


The following lemma and corollary can be found in \cite{HS98} (page 174).

\begin{lemma} Let $n \in \{1,2,\dots,\}$ and $N = 2n+1$, then:
\begin{enumerate}
\item $\mathcal{S}_N$ has among its fixed points $\mathbf{e}_i \in \Delta_N$ ($i=1,\dots,N$) and $\mathbf{u}^*=\frac{1}{N}\mathbf{1}$ in the interior of $\Delta_N$.

\item Furthermore if $\frac{1}{N}\mathbf{1}$ is stable, it is globally asymptotically stable on the interior of $\Delta_N$. If $\frac{1}{N}\mathbf{1}$ is unstable, then all trajectories converge to the boundary of $\Delta_N$ unless $\mathbf{u}_0 = \mathbf{u}^* = \frac{1}{N}\mathbf{1}$.
\end{enumerate}
\hfill\qed
\end{lemma}
%
\begin{corollary} The fixed point $\mathbf{u}^* = \frac{1}{N}\mathbf{1}$ is the unique interior fixed point for $\mathcal{S}_N$.\hfill\qed
\label{cor:uu}
\end{corollary}
Let $\mathbf{F},\mathbf{G}:\Delta_N \rightarrow \mathbb{R}^n$ be defined component-wise as:
\begin{align}
\mathbf{F}_i(\mathbf{u}) &= u_i\left((\mathbf{e}_i - \mathbf{u})\mathbf{L}_N\mathbf{u}\right)\label{eqn:Fz},\\
\mathbf{G}_i(\mathbf{u}) &= u_i\left((\mathbf{e}_i - \mathbf{u})\mathbf{M}_N\mathbf{u}\right)\label{eqn:Gz}.
\end{align}
The replicator dynamics are then:
\begin{displaymath}
\dot{\mathbf{u}} = \mathbf{F}(\mathbf{u}) + \gamma\mathbf{G}(u),
\end{displaymath}
which are the dynamics that will be used in the control problem of interest.

We note that $\mathbf{F}$ and $\mathbf{G}$ are the functional imprints of the standard payoff matrix $\mathbf{L}_N$ and the actuation matrix $\mathbf{M}_N$ within the replicator framework. Understanding their Jacobian matrices is essential in understanding the dynamics. The structure of the Jacobian matrix for a circulant matrix is given in \cite{HS98} (page 173). However, we need an explicit form not presented in that text. The proof of the following lemmas is given in Appendix \ref{sec:TechProofs}.

\begin{lemma} The Jacobian matrix of $\mathbf{F}$ evaluated at $\mathbf{u}^* = \frac{1}{N}\mathbf{1}$ is:
\begin{displaymath}
\mathbf{J} \triangleq D_\mathbf{u}\mathbf{F} = \frac{1}{N}\mathbf{L}_N.
\end{displaymath}
Consequently, $\mathbf{J}$ is a circulant matrix.
\hfill\qed
\label{lem:DF}
\end{lemma}

\begin{lemma} The Jacobian matrix of $\mathbf{G}$ evaluated at $\mathbf{u}^* = \frac{1}{N}\mathbf{1}$ is:
\begin{displaymath}
\mathbf{H} \triangleq D_\mathbf{u}\mathbf{G} = \frac{1}{N^2}\mathbf{M}_N - \frac{2n}{N^2}\left(\mathbf{1}_N - \mathbf{M}_N\right) = \frac{1}{N^2}\left(N\left(\mathbf{M}_N-\mathbf{1}_N\right) + \mathbf{1}_N\right),
\end{displaymath}
where $\mathbf{1}_N$ is an $N \times N$ matrix of $1$'s. Consequently $\mathbf{H}$ is a circulant matrix.
\hfill\qed
\label{lem:DG}
\end{lemma}

\begin{theorem} If $\gamma > 0$, then the fixed point $\mathbf{u}^* = \frac{1}{N}\mathbf{1}$ is asymptotically stable. If $\gamma < 0$, then the fixed point $\mathbf{u}^*$ is asymptotically unstable.
\label{thm:GammaSign}
\end{theorem}
\begin{proof} From Lemmas \ref{lem:DF} and \ref{lem:DG}, we note that the Jacobian matrix of $\mathcal{S}_N$ at $\mathbf{u}^*$ has the following form:
\begin{displaymath}
\bm{\mathcal{J}} = \mathbf{J} + \gamma\mathbf{H} = \frac{1}{N}\mathbf{L}_N + \gamma\left(\frac{1}{N^2}\mathbf{M}_N - \frac{2n}{N^2}\left(\mathbf{1}_N - \mathbf{M}_N\right)\right).
\end{displaymath}
By its construction, it is a circulant matrix with first row given by:
\begin{equation}
\bm{\mathcal{J}}_{1j} =
\begin{cases}
- \gamma\frac{2n}{N^2} & \text{if $j = 1$},\\
-\frac{1}{N}-\gamma\frac{2n}{N^2} & \textit{if $j > 1$ and $j-1$ is odd},\\
\frac{1}{N} + \gamma\frac{1}{N^2} & \text{otherwise}.
\end{cases}
\label{eqn:JacobianRow1}
\end{equation}
Letting $\omega_j$ for $j = 0,\dots,N-1$ be the $N^\text{th}$ roots of unity\footnote{For details see \cite{D12}}, we know that the $j^{th}$ eigenvalue of $\bm{\mathcal{J}}$ is:
\begin{displaymath}
\sum_{k=1}^N\bm{\mathcal{J}}_{1,k}\omega_j^{k-1}.
\end{displaymath}
It now remains to show that the sign of the real-part of $\lambda_j$ is entirely dependent on $\gamma$. The real part of the eigenvalue is given by:
\begin{equation}
\mathrm{Re}\left(\lambda_j\right) =\sum_{k=1}^N \bm{\mathcal{J}}_{1,k}\cos\left(\frac{2\pi j(k-1)}{N}\right).
\label{eqn:EigReal}
\end{equation}
The first eigenvalue ($j=0$) is real and readily computed:
\begin{displaymath}
\lambda_0 = n\left(\gamma\frac{1}{N^2} - \gamma\frac{2n}{N^2} \right) - \gamma\frac{2n}{N^2} = -\gamma\frac{n(1+2n)}{N^2}.
\end{displaymath}
It is clear at once that the sign of this eigenvalue is entirely controlled by the sign of $\gamma$.

For $j > 0$, note that the periodicity of the cosine function (and the fact that the roots of unity are the vertices of the regular unit $N$-gon) implies that the coefficient of $\bm{\mathcal{J}}_{1,k}$ is identical to the coefficient of $\bm{\mathcal{J}}_{1,N-(k-2)}$ if $2 \leq k \leq n+1$. From this fact and Expression \ref{eqn:JacobianRow1}, the sum in Equation \ref{eqn:EigReal} becomes:
\begin{displaymath}
\mathrm{Re}\left(\lambda_j\right) = \gamma\left(-\frac{2n}{N^2}+\sum_{k=2}^{n+1}\cos\left(\frac{2\pi j(k-1)}{N}\right)\left(\frac{1}{N^2}-\frac{2n}{N^2} \right)\right).
\end{displaymath}
Factoring further we see:
\begin{multline*}
\mathrm{Re}\left(\lambda_j\right) = \frac{\gamma}{N^2}\left(-2n +
\sum_{k=2}^{n+1}\cos\left(\frac{2\pi j(k-1)}{N}\right)\left(1-2n\right)
\right) = \\ \frac{\gamma}{N^2}\left(-2n + \left(1-2n\right)\left( \sum_{k=2}^{n+1}\cos\left(\frac{2\pi j(k-1)}{N}\right)
\right) \right).
\end{multline*}
The roots of unity are evenly distributed on the vertices of the unit $N$-gon in $\mathbb{C}$ and therefore the sum of the real parts must be zero. It follows that:
\begin{displaymath}
\sum_{k=2}^{n+1}\cos\left(\frac{2\pi j(k-1)}{N}\right) = -\frac{1}{2}.
\end{displaymath}
We now obtain an exact value for the real parts of the eigenvalues:
\begin{equation}
\mathrm{Re}\left(\lambda_j\right) = \frac{\gamma}{N^2}\left(-2n - \frac{1}{2}\left(1-2n\right)\right) = -\frac{\gamma}{N^2}\left(n+\frac{1}{2}\right).
\end{equation}
Thus we have proved that when $\gamma > 0$, then $\mathrm{Re}(\lambda_j) < 0$ for all $j$ and if $\gamma < 0$, then $\mathrm{Re}(\lambda_j) > 0$. The asymptotic stability (resp. instability) of the fixed point follows immediately.
\end{proof}

\section{The Control Problem and Some General Results}\label{sec:GeneralControl}
We now state our control problem of interest:
\begin{equation}
\mathcal{C}_N = \left\{
\begin{aligned}
\min \;\; & \int_0^{t_f}\;\frac{1}{2}\norm{\mathbf{u} - \mathbf{u}^*}^2 + \frac{r}{2}\gamma^2\;\;dt\\
s.t.\;\; &\dot{\mathbf{u}} = \mathbf{F}(\mathbf{u}) + \gamma\mathbf{G}(\mathbf{u}),\\
&\mathbf{u}(0) = \mathbf{u}_0.
\end{aligned}\right.
\label{eqn:MainControl}
\end{equation}
where $\mathbf{u}^* = \tfrac{1}{N}\mathbf{1}$. Such a problem arises naturally if we consider agents interacting in a cyclic manner and $\gamma$ is a costly control mechanism, i.e., a penalty by which a benevolent social planner may control species populations. Because such a control mechanism is inefficient, the controller seeks to minimize the penalty while driving the population toward a mixed state. As in \cite{GB17}, we will show that a quasi-linearization of this control problem has special structure, and we extend our results to show that this special structure holds for all odd cyclic games (i.e., for all $N = 2n+1$). Furthermore, we discuss the limiting behavior of the control as $N$ grows large. To do this, we first consider a very general optimal control problem and obtain necessary conditions for simplifying the Euler-Lagrange necessary conditions. We then use these simplifications to generalize the results in \cite{GB17}. For the interested reader, an overview of optimal control necessary and sufficient conditions can be found in \cite{Mang66,Ki04,Friesz10} with the more modern Geometric Optimal Control found in \cite{SL12}. We provide a summary of the elementary results from optimal control theory used in the remainder of this paper in Appendix \ref{sec:OptimalControl}.

\subsection{Control Problems with One Control and Vanishing Lie Bracket}
In the remainder of this section, the functions $\mathbf{F},\mathbf{G}:\mathbb{R}^n \rightarrow \mathbb{R}^n$ are arbitrary smooth functions, rather than the functions specific to the replicator dynamics for cyclic games given in Equations \ref{eqn:Fz} and \ref{eqn:Gz}, $\mathbf{x} \in \mathbb{R}^n$ is a state vector, and $\gamma$ is the control function to be determined.

Consider the optimal control problem with form:
\begin{equation}
\left\{
\begin{aligned}
\min\;\; & \Psi(\mathbf{x}(t_f)) + \int_0^{t_f} F_0(\mathbf{x}) + \gamma G_0(\mathbf{x}) + \frac{r}{2} \gamma^2 \;\; dt\\
s.t.\;\; & \dot{\mathbf{x}} = \mathbf{F}(\mathbf{x}) + \gamma\mathbf{G}(\mathbf{x}),\\
& \mathbf{x}(0) = \mathbf{x}_0.
\end{aligned}
\right.
\label{eqn:main}
\end{equation}
The functions $F_0,G_0:\mathbb{R}^n \rightarrow \mathbb{R}$ are smooth. Let $r > 0$, $t_f$ be the terminal time, and $F_0(\mathbf{x})$ be convex. Expression \ref{eqn:MainControl} has this structure, so we are simply considering a more general case of our problem of interest.

The Euler-Lagrange necessary conditions for control are simple to derive for this problem and have an almost linear behavior. Note the Hamiltonian is:
\begin{equation}
\mathcal{H}(\mathbf{x},\gamma,\bm{\lambda}) = F_0(\mathbf{x}) +\gamma G_0(\mathbf{x}) + \frac{r}{2}\gamma^2 + \bm{\lambda}^T\mathbf{F}(\mathbf{x})  + \gamma \bm{\lambda}^T\mathbf{G}(\mathbf{x}).
\label{eqn:Hamiltonian}
\end{equation}
The Hamiltonian is (strictly) convex in the control $\gamma$, and thus we propose the following:
\begin{lemma} Any solution $\gamma^*$ to $\mathcal{H}_\gamma = 0$ satisfies the necessary conditions:
\begin{enumerate}
\item $\mathcal{H}_\gamma = 0$, and
\item $\mathcal{H}_{\gamma\gamma} > 0$, the strong Legendre-Clebsch condition;
\end{enumerate}
therefore, it minimizes the Hamiltonian at all times. \hfill\qed
\label{lem:NecessaryControl}
\end{lemma}
Deriving the optimal control by solving $\partial \mathcal{H}/\partial{\gamma} = 0$ for $\gamma$ to obtain:
\begin{equation}
\gamma^* = -\frac{1}{r}\left(\bm{\lambda}^T\mathbf{G}(\mathbf{x})+G_0(\mathbf{x})\right).
\label{eqn:ClosedFormCoState}
\end{equation}
The two conditions in Lemma \ref{lem:NecessaryControl}, along with the fact that $\mathbf{x}^*$ and $\bm{\lambda}^*$ solve the resulting Euler-Lagrange two-point boundary value problem (see Expression \ref{eqn:EulerLagrange}), form the complete set of necessary conditions for the optimal control problem. Adding in the additional requirement that the corresponding matrix Riccati equation is bounded on $[0,t_f]$, these form sufficient conditions for a weak local minimal optimal controller \cite{BY65,J70}. We discuss this sufficient condition in the sequel.

For simplicity, we refer to the optimal control as $\gamma$ (rather than $\gamma^*$) in the remainder of this paper and assume it is given by Equation \ref{eqn:ClosedFormCoState}. The adjoint dynamics are:
\begin{equation}
\dot{\bm{\lambda}}^T = -(\nabla_\mathbf{x}F_0(\mathbf{x}))^T - \gamma(\nabla_\mathbf{x}G_0(\mathbf{x}))^T -
\bm{\lambda}^TD_\mathbf{x}\mathbf{F} - \gamma \bm{\lambda}^TD_\mathbf{x}\mathbf{G},
\label{eqn:costate}
\end{equation}
where $D_\mathbf{x}\mathbf{F}$ is the Jacobian (with respect to $\mathbf{x}$).
Thus we have the Euler-Lagrange two-point boundary value problem:
\begin{equation}
\left\{
\begin{aligned}
\dot{\mathbf{x}} &= \mathbf{F}(\mathbf{x}) + \gamma\mathbf{G}(\mathbf{x}),\\
\dot{\bm{\lambda}} &= -\nabla_\mathbf{x}F_0(\mathbf{x})-\gamma\nabla_\mathbf{x}G_0(\mathbf{x}) -(D_\mathbf{x}\mathbf{F})^T\bm{\lambda} - u(D_\mathbf{x}\mathbf{G})^T\bm{\lambda},\\
\mathbf{x}(0) &= \mathbf{x}_0,\\
\bm{\lambda}(t_f) &= \nabla_\mathbf{x}\Psi(\mathbf{x}[t_f]). \qquad \text{(Transverality Condition)}
\end{aligned}
\right.
\label{eqn:EulerLagrange}
\end{equation}
\begin{proposition} If $\gamma$ is an optimal control, then:
\begin{equation}
\gamma(t_f) = -\frac{1}{r}\left(\nabla_{\mathbf{x}}\Psi(\mathbf{x}(t_f))^T\mathbf{G}(\mathbf{x}(t_f)) + G_0(\mathbf{x}(t_f)) \right).
\end{equation}
\end{proposition}
\begin{proof} This follows from the transversality condition.
\end{proof}
From Equation \ref{eqn:ClosedFormCoState}, note that:
\begin{equation}
r\dot{\gamma} = -\dot{\bm{\lambda}}^T\mathbf{G}(\mathbf{x}) - \bm{\lambda}^T(D_\mathbf{x}\mathbf{G})\dot{\mathbf{x}} - (\nabla_\mathbf{x}G_0)\dot{\mathbf{x}}.
\end{equation}
Then:
\begin{multline}
r\dot{\gamma} = \left((\nabla_\mathbf{x}F_0(\mathbf{x}))^T + \gamma(\nabla_\mathbf{x}G_0(\mathbf{x}))^T+
\bm{\lambda}^T(D_\mathbf{x}\mathbf{F}) + \gamma \bm{\lambda}^T(D_\mathbf{x}\mathbf{G})\right)\mathbf{G}(\mathbf{x}) - \\
\bm{\lambda}^T(D_\mathbf{x}\mathbf{G})\left(\mathbf{F}(\mathbf{x}) + \gamma\mathbf{G}(\mathbf{x}) \right) -  (\nabla_\mathbf{x}G_0)\left(\mathbf{F}(\mathbf{x}) + \gamma\mathbf{G}(\mathbf{x}) \right).
\end{multline}
Simplifying we have:
\begin{displaymath}
r\dot{\gamma} = (\nabla_\mathbf{x}F_0(\mathbf{x}))^T\mathbf{G}(\mathbf{x}) - (\nabla_\mathbf{x}G_0(\mathbf{x}))^T\mathbf{F}(\mathbf{x}) + \bm{\lambda}^T\left((D_\mathbf{x}\mathbf{F})\mathbf{G}(\mathbf{x}) - (D_\mathbf{x}\mathbf{G})\mathbf{F}(x)\right).
\end{displaymath}
If the Lie Bracket vanishes, i.e.,:
\begin{equation}
[\mathbf{F},\mathbf{G}]=(D_\mathbf{x}\mathbf{F})\mathbf{G} - (D_\mathbf{x}\mathbf{G})\mathbf{F} = \mathbf{0},
\label{eqn:LieBracketVanish}
\end{equation}
then this simplifies to:
\begin{equation}
\dot{\gamma} = \frac{1}{r}\left(  (\nabla_\mathbf{x}F_0(\mathbf{x}))^T\mathbf{G}(\mathbf{x}) - (\nabla_\mathbf{x}G_0(\mathbf{x}))^T\mathbf{F}(\mathbf{x})  \right),
\end{equation}
and all co-state variables are eliminated. We have shown the following:
\begin{theorem} Consider the general optimal control problem given in Expression \ref{eqn:main}. If $[\mathbf{F},\mathbf{G}] = \mathbf{0}$ and $\gamma$ is an optimal control, then the pair $(\mathbf{x}(\gamma), \gamma)$ is a solution of the two point boundary value problem:
\begin{equation}
\left\{
\begin{aligned}
\dot{\mathbf{x}} = &\mathbf{F}(\mathbf{x}) + \gamma\mathbf{G}(\mathbf{x}),\\
\dot{\gamma} = & \frac{1}{r}\left(  (\nabla_\mathbf{x}F_0(\mathbf{x}))^T\mathbf{G}(\mathbf{x}) - (\nabla_\mathbf{x}G_0(\mathbf{x}))^T\mathbf{F}(\mathbf{x})  \right),\\
\mathbf{x}(0) =& \mathbf{x}_0,\\
\gamma(t_f) = & -\frac{1}{r}\left(\nabla_{\mathbf{x}}\Psi(\mathbf{x}(t_f))^T\mathbf{G}(\mathbf{x}(t_f)) + G_0(\mathbf{x}(t_f)) \right).
\end{aligned}\right.
\label{eqn:VectorThm}
\end{equation}\hfill\qed
\label{thm:main}
\end{theorem}
Geometrically, Equation \ref{eqn:LieBracketVanish} implies that the flows derived by the vector fields $\mathbf{F}$ and $\mathbf{G}$ commute locally. From a game-theoretic view, this means that locally evolutionary motion caused by competition in the uncontrolled game  commutes with evolutionary motion caused by the actuation payoffs on local space/time scales. As we see in the sequel, this is not true for actuated cyclic games, but is true for their quasi-linear approximations as in \cite{GB17}, meaning we can use Theorem \ref{thm:main} to determine properties of the optimal control near the interior fixed point. 

It is worth noting that a differential equation for the control function is derived in \cite{B78}, without the assumption of the vanishing Lie Bracket. However, without this assumption the system does not simplify in as useful a way and, in fact, in \cite{B78} the relevant Lie Bracket is not considered. Note, in formulating Theorem \ref{thm:main}, we are assuming that solving the Euler-Lagrange equations will yield an optimal control. We can use the well known fact that a sufficient condition for optimality is the boundedness of the solution to the matrix Ricatti equation \cite{BY65,J70} to derive complete necessary and sufficient conditions for optimality of the control. Let:
\begin{displaymath}
\dot{\mathbf{x}} = \mathbf{F}(\mathbf{x}) + \gamma\mathbf{G}(\mathbf{x}) = \mathbf{f}(\mathbf{x},\gamma).
\end{displaymath}
Then the Matrix Ricatti equation is:
\begin{equation}
\left\{
\begin{aligned}
-\dot{\mathbf{S}} &= D_\mathbf{xx}\mathcal{H} + (D_\mathbf{x}\mathbf{f})^T\mathbf{S} + \mathbf{S}(D_\mathbf{x}\mathbf{f}) - \\
& \hspace*{4em}\frac{1}{r}\left(D_{\gamma\mathbf{x}}\mathcal{H} + (\partial_\gamma\mathbf{f})^T\mathbf{S}\right)^T\left(D_{\gamma\mathbf{x}}\mathcal{H} + (\partial_\gamma\mathbf{f})^T\mathbf{S}\right),\\
\mathbf{S}(t_f) &= \nabla^2_\mathbf{x}\Psi(\mathbf{x}[t_f]).
\end{aligned}\right.
\label{eqn:matrixricatti}
\end{equation}
Here $D_\mathbf{xx}$ is the second order differential operator with respect to the state and $\partial_\gamma$ is an ordinary partial derivative, since there is only one control variable. When taken together with Lemma \ref{lem:NecessaryControl}, the system of differential equations in Theorem \ref{thm:main} and the co-state dynamics, Equation \ref{eqn:costate}, we have a complete characterization of the necessary and sufficient conditions for the optimal control. This yields the corollary:
\begin{corollary}[Corollary to Theorem \ref{thm:main}] Let $\mathbf{f}(\mathbf{x},\gamma) = \mathbf{F}(\mathbf{x}) + \gamma\mathbf{G}(\mathbf{x})$ in the optimal control problem in Expression \ref{eqn:main}, with Hamiltonian $\mathbf{H}(\mathbf{x},\gamma,\bm{\lambda})$. Assume $[\mathbf{F},\mathbf{G}] = \mathbf{0}$. Any solution to the system of differential equations:
\begin{equation}
\left\{
\begin{aligned}
\dot{\mathbf{x}} &= \mathbf{f}(\mathbf{x},\gamma),\\
\dot{\gamma} &=  \frac{1}{r}\left(  (\nabla_\mathbf{x}F_0(\mathbf{x}))^T\mathbf{G}(\mathbf{x}) - (\nabla_\mathbf{x}G_0(\mathbf{x}))^T\mathbf{F}(\mathbf{x})  \right),\\
\dot{\bm{\lambda}} &= -\nabla_\mathbf{x}F_0(\mathbf{x})-\gamma\nabla_\mathbf{x}G_0(\mathbf{x}) -(D_\mathbf{x}\mathbf{F})^T\bm{\lambda} - u(D_\mathbf{x}\mathbf{G})^T\bm{\lambda},\\
-\dot{\mathbf{S}} &= D_\mathbf{xx}\mathcal{H} + (D_\mathbf{x}\mathbf{f})^T\mathbf{S} + \mathbf{S}(D_\mathbf{x}\mathbf{f}) - \\
& \hspace*{4em}\frac{1}{r}\left(D_{\gamma\mathbf{x}}\mathcal{H} + (\partial_\gamma\mathbf{f})^T\mathbf{S}\right)^T\left(D_{\gamma\mathbf{x}}\mathcal{H} + (\partial_\gamma\mathbf{f})^T\mathbf{S}\right),\\
\mathbf{x}(0) &= \mathbf{x}_0,\\
\gamma(t_f) &=  -\frac{1}{r}\left(\nabla_{\mathbf{x}}\Psi(\mathbf{x}(t_f))^T\mathbf{G}(\mathbf{x}(t_f)) + G_0(\mathbf{x}(t_f)) \right),\\
\bm{\lambda}(t_f) &= \nabla_x\Psi(\mathbf{x}[t_f]),\\
\mathbf{S}(t_f) &= \nabla^2_\mathbf{x}\Psi(\mathbf{x}[t_f]),
\end{aligned}\right.
\end{equation}
in which $\gamma(t) = -\tfrac{1}{r}\left(\bm{\lambda}^T\mathbf{G}[\mathbf{x}] + G_0[\mathbf{x}]\right)$ and $\mathbf{S}$ is bounded for all $t \in [0,t_f]$ constitutes a weak local optimal solution for Expression \ref{eqn:main}.
\label{cor:main}
\end{corollary}
We note that this is the general analog of Proposition 2 in \cite{FG17}, which is specialized to a control problem with one-dimensional state. In general, checking the boundedness of the solution to the Matrix Ricatti equation must be done numerically. In the sequel we develop a simpler test for optimality using Mangasarian's sufficiency condition; i.e., by checking that the Hamiltonian is jointly convex.

Problem \ref{eqn:MainControl} ($\mathcal{C}_N$) does not satisfy the necessary condition that $[\mathbf{F},\mathbf{G}] = \mathbf{0}$. However, a quasi-linearization of the problem does satisfy this condition (as in \cite{GB17}). We now discuss a special case of Theorem \ref{thm:main} as well as extensions that apply to this quasi-linearized form.

\subsection{The Quasi-Linear Case}
\label{sec:linear}
For the remainder of this section, let $\mathbf{J}$ and $\mathbf{H}$ be arbitrary matrices of appropriate size, rather than the Jacobian matrices derived in Lemmas \ref{lem:DF} and \ref{lem:DG}.

In Problem \ref{eqn:main}, let:
\begin{equation}
\left\{
\begin{aligned}
F_0(\mathbf{x}) &= \frac{1}{2}\mathbf{x}^T\mathbf{Q}\mathbf{x},\\
G_0(\mathbf{x}) &= 0,\\
\mathbf{F}(\mathbf{x}) &= \mathbf{J}\mathbf{x},\\
\mathbf{G}(\mathbf{x}) &= \mathbf{H}\mathbf{x}.
\end{aligned}
\right.
\label{eqn:LinearConditions}
\end{equation}
where $\mathbf{Q}$ is a (symmetric) positive definite matrix of appropriate size. We will add additional criteria to $\mathbf{J}$ and $\mathbf{H}$ as we proceed. We refer to this as a quasi-linear case because the only non-linearity arises from the interaction of the state and control variables. The following Corollary is immediate from Theorem \ref{thm:main}:
\begin{corollary} If the identify $\mathbf{J}\mathbf{H} = \mathbf{H}\mathbf{J}$ holds and $\gamma^*$ is an optimal control, then the pair $(\mathbf{x}(\gamma^*), \gamma^*)$ is a solution of the two point boundary value problem:
\begin{equation}
\left\{
\begin{aligned}
\dot{\mathbf{x}} = &\mathbf{J}\mathbf{x} + \gamma\mathbf{H}\mathbf{x},\\
\dot{\gamma} = & \frac{1}{r}\mathbf{x}^T\mathbf{Q}\mathbf{H}\mathbf{x},\\
\mathbf{x}(0) =& \mathbf{x}_0,\\
\gamma(t_f) = & -\frac{1}{r}\nabla_{\mathbf{x}}\Psi(\mathbf{x}(t_f))^T\mathbf{H}\mathbf{x}(t_f).
\end{aligned}\right.
\end{equation}
\hfill\qed
\label{cor:LinRes1}
\end{corollary}
The condition that $\mathbf{J}$ and $\mathbf{H}$ commute is exactly the statement that the Lie Bracket of the vector fields in the dynamics vanishes. Therefore, Theorem \ref{thm:main} can be applied to any linear quadratic control problem where the state equation satisfies this condition.

We now derive some special results on $\ddot{\gamma}$ and the optimal control in this quasi-linear case. Let $\mathbf{K} \triangleq \mathbf{Q}\mathbf{H}$ and assume $\mathbf{J}\mathbf{H} = \mathbf{H}\mathbf{J}$. Computing the second derivative of $\gamma$ yields:
\begin{multline*}
r \ddot{\gamma} = \dot{\mathbf{x}}^T\mathbf{K}\mathbf{x} + \mathbf{x}^T\mathbf{K}\dot{\mathbf{x}} =
\left(\mathbf{x}^T\mathbf{J}^T + \gamma\mathbf{x}^T\mathbf{H}^T\right)\mathbf{K}\mathbf{x} + \mathbf{x}^T\mathbf{K}\left(\mathbf{J}\mathbf{x} + \gamma \mathbf{H}\mathbf{x}\right) = \\
\mathbf{x}^T\left(\mathbf{J}^T\mathbf{K} + \mathbf{K}\mathbf{J} + \gamma\left( \mathbf{H}^T\mathbf{K} + \mathbf{K}\mathbf{H} \right)
\right)\mathbf{x}.
\end{multline*}
To simplify this, we will add an additional assumption to $\mathbf{J}$; suppose that $\mathbf{J}^T = -\mathbf{J}$ (i.e., $\mathbf{J}$ is skew-symmetric) and $\mathbf{J}\mathbf{K} = \mathbf{K}\mathbf{J}$. Then:
\begin{equation}
r \frac{\ddot{\gamma}}{\gamma} = \mathbf{x}^T\left( \mathbf{H}^T\mathbf{K} + \mathbf{K}\mathbf{H} \right)\mathbf{x}.
\end{equation}

Before proceeding note that:
\begin{align*}
\frac{d}{dt}\left(\mathbf{x}^T\mathbf{Q}\mathbf{x}\right) = &
\left(\mathbf{x}^T\mathbf{J}^T + \gamma\mathbf{x}^T\mathbf{H}^T \right)\mathbf{Q}\mathbf{x} + \mathbf{x}^T\mathbf{Q}\left(\mathbf{J}\mathbf{x} + \gamma\mathbf{H}\mathbf{x}\right) = \\
&\mathbf{x}^T\left(-\mathbf{J}\mathbf{Q} + \mathbf{Q}\mathbf{J}\right)\mathbf{x} + \gamma\mathbf{x}^T\left(\mathbf{H}^T\mathbf{Q} + \mathbf{Q}\mathbf{H}\right)\mathbf{x} =\\
&\mathbf{x}^T\left(-\mathbf{J}\mathbf{Q} + \mathbf{Q}\mathbf{J}\right)\mathbf{x} + \gamma \mathbf{x}^T\left(\mathbf{H}^T\mathbf{Q}^T + \mathbf{Q}\mathbf{H}\right)\mathbf{x} = \\
&\mathbf{x}^T\left(-\mathbf{J}\mathbf{Q} + \mathbf{Q}\mathbf{J}\right)\mathbf{x} + 2 \gamma \mathbf{x}^T\mathbf{K}\mathbf{x} = \\
&\mathbf{x}^T\left(-\mathbf{J}\mathbf{Q} + \mathbf{Q}\mathbf{J}\right)\mathbf{x} + 2 r\gamma \dot{\gamma}.
\end{align*}
Thus, we have the following proposition and its corollary:
\begin{proposition} If $\mathbf{J} = -\mathbf{J}^T$, $\mathbf{J}\mathbf{H} = \mathbf{H}\mathbf{J}$ and $\mathbf{J}\mathbf{Q} = \mathbf{Q}\mathbf{J}$ and $\mathbf{K} \triangleq \mathbf{Q}\mathbf{H}$, then:
\begin{enumerate}
\item \begin{equation}\nonumber
\mathbf{J}\mathbf{K} = \mathbf{K}\mathbf{J},
\end{equation}
\item
\begin{equation}
\frac{d}{dt}\left(\mathbf{x}^T\mathbf{Q}\mathbf{x}\right) = 2r\gamma \dot{\gamma},
\end{equation}
\item
\begin{equation}
r \frac{\ddot{\gamma}}{\gamma} = \mathbf{x}^T\left(\mathbf{H}^T\mathbf{K} + \mathbf{K}\mathbf{H}\right)\mathbf{x}.
\end{equation}
\end{enumerate}
\label{prop:VectorProp}
\end{proposition}
\begin{corollary} For some constant $C$,
\begin{equation}
\mathbf{x}^T\mathbf{Q}\mathbf{x} = r\gamma^2 + C
\end{equation}
is the implicit closed-loop control, where $C$ must satisfy:
\begin{equation}
C = \mathbf{x}^T(t_f)\mathbf{Q}\mathbf{x}(t_f) - r\left(\frac{1}{r}\nabla_{\mathbf{x}}\Psi(\mathbf{x}(t_f))^T\mathbf{H}\mathbf{x}(t_f) \right)^2.
\label{eqn:C2}
\end{equation}
Furthermore the optimal control $\gamma$ exists at time $t$ just in case:
\begin{equation}
\mathbf{x}^T(t)\mathbf{Q}\mathbf{x}(t) - C \geq 0.
\end{equation}
\label{cor:xQx}
\end{corollary}

\section{Application of Control Results to Complete Odd Circulant Games}\label{sec:CyclicControl}
We now return to the study of cyclic games with $N$ strategies and specifically to the control problem in Expression \ref{eqn:MainControl}. As noted already, we cannot apply Theorem \ref{thm:main} directly to Problem \ref{eqn:MainControl} because the appropriate Lie Bracket does not vanish. However, we can construct the quasi-linearized form of the problem. Let $\mathbf{x} = \mathbf{u} - \mathbf{u}^*$. The quasi-linearized problem is:
\begin{equation}
\tilde{\mathcal{C}}_N = \left\{
\begin{aligned}
\min \;\; & \int_0^{t_f}\;\frac{1}{2}\norm{\mathbf{x}}^2 + \frac{r}{2}\gamma^2\;\;dt\\
s.t.\;\; &\dot{\mathbf{x}} = \mathbf{J}\mathbf{x} + \gamma\mathbf{H}\mathbf{x},\\
&\mathbf{x}(0) = \mathbf{x}_0.
\end{aligned}\right.
\label{eqn:MainControlQuasi}
\end{equation}
In Expression \ref{eqn:MainControlQuasi}, $\mathbf{J}$ and $\mathbf{H}$ are the Jacobian matrices of $\mathbf{F}(\mathbf{u})$ and $\mathbf{G}(\mathbf{u})$ as defined in Lemmas \ref{lem:DF} and \ref{lem:DG}.  Problem \ref{eqn:MainControlQuasi} is an instance of the general control problem studied in Section \ref{sec:GeneralControl}.

The following useful fact follows at once from Lemma \ref{lem:DF}.
\begin{corollary} The Jacobian matrix $\mathbf{J}$ is skew-symmetric. \hfill\qed
\label{cor:SkewSymmetric}
\end{corollary}

\begin{lemma} Let $\gamma$ be the optimal control for Problem \ref{eqn:MainControlQuasi}. Then:
\begin{enumerate}
\item The (open-loop) optimal control obeys the following differential equations:
\begin{align}
\dot{\gamma} &= \frac{1}{r}\mathbf{x}^T\mathbf{H}\mathbf{x},\; \gamma(t_f) = 0, \label{eqn:dotgamma}\\
r\frac{\ddot{\gamma}}{\gamma} &= \mathbf{x}^T\left(\mathbf{H}^T\mathbf{H} + \mathbf{H}\mathbf{H}\right)\mathbf{x}. \label{eqn:ddotgamma1}
\end{align}
\item The following identity holds:
\begin{equation}
\mathbf{x}^T\mathbf{x} = \norm{x}^2 = r\gamma^2 + C,
\label{eqn:closedloop}
\end{equation}
where:
\begin{displaymath}
C = \norm{\mathbf{x}(t_f)}^2.
\end{displaymath}
\end{enumerate}
\label{lem:MainControl1}
\end{lemma}
\begin{proof}
Problem \ref{eqn:MainControlQuasi} is an instance of Problem \ref{eqn:main}, but with quasi-linear system dynamics and quadratic objective as given in the quasi-linear conditions in Expression \ref{eqn:LinearConditions}. In particular, Problem \ref{eqn:MainControlQuasi} sets $\mathbf{Q} = \mathbf{I}_N$. As a consequence the matrix $\mathbf{K} = \mathbf{Q}\mathbf{H} = \mathbf{H}$. From Corollary \ref{cor:SkewSymmetric}, we know $\mathbf{J}$ is skew-symmetric. Further, since the circulant matrices form a commutative algebra, we have $\mathbf{H}\mathbf{J} = \mathbf{J}\mathbf{H}$. The lemma follows at once from Corollary \ref{cor:LinRes1}, Proposition \ref{prop:VectorProp} and Corollary \ref{cor:xQx}.
\end{proof}
Expression \ref{eqn:closedloop} is the closed-loop control law for the controlled cyclic game. Furthermore, Equation \ref{eqn:dotgamma} allows us to determine some structural properties of $\gamma$.

\begin{proposition} The matrix $\mathbf{H}$ is negative definite and therefore $\dot{\gamma} \leq 0$ for all $t$.
\label{prop:DotGammaSign}
\end{proposition}
In addition to determining that $\gamma$ is decreasing, Problem \ref{eqn:MainControlQuasi} has further special structure, which allows us to understand the structure of the derived control $\gamma$ in greater detail and ultimately derive a closed-form approximation for large $N$. The derivation is similar to the one found in \cite{GB17} for a special case diffeomorphic to rock-paper-scissors. The proof of the following lemma is given in Appendix \ref{sec:TechProofs}.
\begin{lemma} 
For all $\mathbf{x} \in \mathbb{R}^N$:
\begin{equation}
\mathbf{x}^T\left(\mathbf{H}^T\mathbf{H} + \mathbf{H}\mathbf{H} + \mathbf{H}\right)\mathbf{x} = -\frac{n}{N^2}\left\lVert\mathbf{x}\right\rVert^2.
\end{equation}
\label{lem:5H}
\end{lemma}

\begin{theorem} If $\gamma$ is the open-loop optimal control for Problem \ref{eqn:MainControlQuasi}, then $\gamma$ satisfies the following second order differential equation:
\begin{equation}
\left\{
\begin{aligned}
r\ddot{\gamma} + r\gamma\dot{\gamma} + \frac{n}{N^2}\gamma\left(r\gamma^2 + C\right) &= 0,\\
\gamma(t_f) &= 0,\\
\gamma'(0) &= \frac{1}{r}\mathbf{x}_0^T\mathbf{H}\mathbf{x}_0,\\
C &= \norm{x(t_f)}^2.
\end{aligned}
\right.
\label{eqn:2ODE-Full}
\end{equation}
\end{theorem}
\begin{proof}
From Lemma \ref{lem:MainControl1} we have:
\begin{displaymath}
r\frac{\ddot{\gamma}}{\gamma} = \mathbf{x}^T\left(\mathbf{H}^T\mathbf{H} + \mathbf{H}\mathbf{H}\right)\mathbf{x},
\end{displaymath}
and
\begin{displaymath}
r\dot{\gamma} = \mathbf{x}^T\mathbf{H}\mathbf{x}.
\end{displaymath}
Adding these together we obtain:
\begin{displaymath}
r\frac{\ddot{\gamma}}{\gamma} + r\dot{\gamma} = \mathbf{x}^T\left(\mathbf{H}^T\mathbf{H} + \mathbf{H}\mathbf{H} + \mathbf{H}\right)\mathbf{x}.
\end{displaymath}
Therefore by Lemma \ref{lem:5H}:
\begin{displaymath}
r\frac{\ddot{\gamma}}{\gamma} + \dot{\gamma} = -\frac{n}{N^2}\norm{x(t)}^2.
\end{displaymath}
From Lemma \ref{lem:MainControl1}, we have:
\begin{displaymath}
r\frac{\ddot{\gamma}}{\gamma} + \dot{\gamma} = -\frac{n}{N^2}\left(r\gamma^2 + C\right),
\end{displaymath}
where $C = \norm{\mathbf{x}(t_f)}^2$. Thus:
\begin{displaymath}
r\ddot{\gamma} + \gamma\dot{\gamma} + \frac{n}{N^2}\gamma\left(r\gamma^2 + C\right) = 0.
\end{displaymath}
The boundary conditions $\gamma(0) = 0$ and $\gamma'(0) = \mathbf{x}_0^T\mathbf{H}\mathbf{x}_0$ follows from Lemma \ref{lem:MainControl1}.
\end{proof}
The following corollary is illustrated in Appendix \ref{app:Examples}.
\begin{corollary} For $N$ large, the open loop control $\gamma$ can be approximated by $\zeta$, a solution to the following two-point boundary value problem:
\begin{equation}
\left\{
\begin{aligned}
r\ddot{\zeta} + r\zeta\dot{\zeta} = 0,\\
\zeta(t_f) &= 0,\\
\zeta'(0) &= \frac{1}{r}\mathbf{x}_0^T\mathbf{H}\mathbf{x}_0.
\end{aligned}
\right.
\label{eqn:2ODE}
\end{equation}
\label{cor:2ODE}
\end{corollary}

\subsection{Closed Form Analysis of the Limiting Behavior}
For simplicity, let $r = 1$ in the following analysis. Corollary \ref{cor:2ODE} can be made more useful by re-writing Equation \ref{eqn:2ODE} as a system of first order differential equations and examining the phase portrait (see Fig. \ref{fig:PhasePortrait}):
\begin{equation}
\begin{aligned}
\dot{\zeta} &= v,\\
\dot{v} &= -\zeta v,\\
\zeta(t_f) &= 0,\\
v(0) &= \mathbf{x}_0^T\mathbf{H}\mathbf{x}_0.
\end{aligned}
\label{eqn:SysODE}
\end{equation}
\begin{figure}[htbp]
\centering
\includegraphics[scale=0.5]{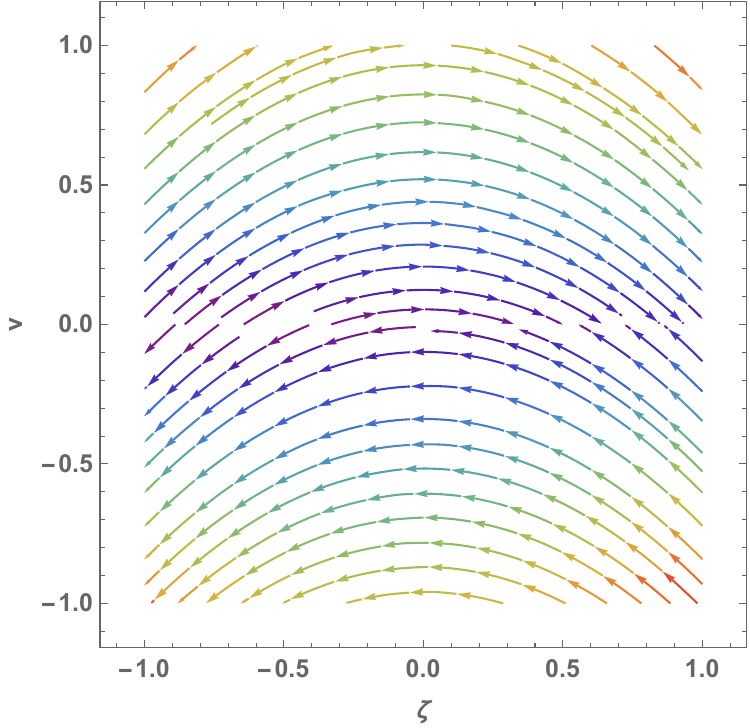}
\caption{The phase portrait of the first order system representing the limiting behavior of the open loop control $\gamma$ and it's first derivative.}
\label{fig:PhasePortrait}
\end{figure}
The phase portrait indicates a sharp behavioral change in the direction field when moving from the $v < 0$ half-plane to the $v > 0$ half-plane. For the half-plane where $v < 0$, $\gamma \geq 0$ necessarily by Theorem \ref{thm:GammaSign}. This is consistent with Proposition \ref{prop:DotGammaSign}.

System \ref{eqn:SysODE} has a closed form solution\footnote{Derived with \textit{Mathematica}$^\text{\textsc{TM}}$.} with several branches. The relevant solutions on the interval $[0,t_f]$ are:
\begin{align*}
\zeta(t) &= \sqrt{2} \sqrt{\kappa  \tanh ^2\left(\frac{\sqrt{\kappa }
   \left(t_f-t\right)}{\sqrt{2}}\right)},\\
\zeta'(t) &= -2 \sqrt{\kappa } \sqrt{\kappa  \tanh ^2\left(\frac{\sqrt{\kappa }
   \left(t_f-t\right)}{\sqrt{2}}\right)} \csch\left(\sqrt{2} \sqrt{\kappa }
   \left(t_f-t\right)\right).
\end{align*}
Note, $\zeta \sim O(|\tanh(\cdot)|)$. Thus, the usual behavior of $\tanh$ is modified so that $\zeta$ is a decreasing function on $[0,t_f]$ and then an increasing function outside this range. As a consequence, this solution is only valid on the control domain of interest.

In the closed form solution, $\kappa$ is a constant of integration that must be chosen so that $v(0) = \zeta'(0) = \mathbf{x}_0^T\mathbf{H}\mathbf{x}_0$. Finding a closed form expression for $\kappa$ is difficult. However as $N$ increases, $\mathbf{x}_0^T\mathbf{H}\mathbf{x}_0$ decreases in size because $\mathbf{H} \sim 1/N$ and $\norm{\mathbf{x}}$ is bounded, since $\mathbf{x}$ is just a translation of $\mathbf{u} \in \Delta_N$. For small values of $\mathbf{x}_0^T\mathbf{H}\mathbf{x}_0$, we expect $\kappa$ to be small because of the structure of $\zeta'(t)$. Furthermore, $\zeta'(0)$ can be approximated as:
\begin{displaymath}
\zeta'(0) \approx -\kappa + O(\kappa^2).
\end{displaymath}
Thus, setting $\kappa = -\mathbf{x}_0^T\mathbf{H}\mathbf{x}_0$ will give a reasonable approximation of the solution. This is illustrated in Appendix \ref{app:Examples}.

\subsection{Sufficiency of the Euler-Lagrange Conditions}\label{sec:Sufficiency}
Corollary \ref{cor:main} contains both necessary and sufficient conditions for the computed $\gamma(t)$ to be the optimal control. However, these conditions require the solution of the matrix Riccati equation. For the quasi-linearized optimal control problem on cyclic games, a simpler test can be constructed using Mangaserian's condition \cite{Mang66}, which states that the Hessian of the Hamiltonian must be positive definite (i.e., jointly convex in state and control). For the optimal control in quasi-linearized cyclic games, the Hamiltonian of this optimal control problem is:
\begin{displaymath}
\mathcal{H}(\mathbf{x},\gamma,\bm{\lambda}) = \norm{\mathbf{x}}^2 + \frac{1}{r}\gamma^2 + \bm{\lambda}^T\mathbf{J}\mathbf{x} + \gamma\bm{\lambda}^T\mathbf{H}\mathbf{x}.
\end{displaymath}
This is a specialization of Equation \ref{eqn:Hamiltonian} to the quasi-linearized cyclic games problem. The Hessian of $\mathcal{H}$ is:
\begin{displaymath}
\mathfrak{H} = \begin{bmatrix}
\mathbf{I}_N & \mathbf{H}^T\bm{\lambda}^T\\
\bm{\lambda}\mathbf{H}& r
\end{bmatrix}.
\end{displaymath}
Here, $\bm{\lambda}$ is the co-state for the optimal control problem.
\begin{theorem} If $r > \norm{\mathbf{H}^T\bm{\lambda}^T}^2$ for all $t \in [0,t_f]$, then the control derived in Lemma \ref{lem:MainControl1} is optimal.
\label{thm:Sufficient}
\end{theorem}
\begin{proof} The Hessian matrix $\mathfrak{H}$ is positive definite if and only if it has a Cholesky decomposition, which then implies that $\mathcal{H}(\mathbf{x},\gamma,\bm{\lambda})$ is convex in both its state and control. Computing the Cholesky decomposition for $\mathfrak{H}$ we obtain:
\begin{displaymath}
\mathfrak{H} = \begin{bmatrix}
\mathbf{I}_N & \mathbf{0}\\
\bm{\lambda}^T\mathbf{H} & \sqrt{r - \norm{\mathbf{H}^T\bm{\lambda}^T}}^2
\end{bmatrix}\begin{bmatrix}
\mathbf{I}_N & \mathbf{H}^T\bm{\lambda}^T\\
\mathbf{0} & \sqrt{r - \norm{\mathbf{H}^T\bm{\lambda}^T}^2}
\end{bmatrix}.
\end{displaymath}
This decomposition exists if and only if $r > \norm{\mathbf{H}^T\bm{\lambda}^T}^2$. The result follows immediately.
\end{proof}
This sufficient condition for optimality is precisely the one identified in \cite{GB17} for the triple public goods game, which was shown to be diffeomorphic to the cyclic game with three strategies (rock-paper-scissors). In Appendix \ref{app:Examples} we show several instances where this sufficient condition is satisfied and one where the matrix Ricatti equation (see Corollary \ref{cor:main}) must be used to establish weak local optimality.

\section{Final Discussion and Future Directions}\label{sec:Conclusion}
In this paper we studied an optimal control problem arising from the class of complete, odd circulant games that generalize rock-paper-scissors. We used the replicator dynamics as the natural equations of motion in the optimal control problem. In particular, the control problem was to drive trajectories toward the unique interior fixed point of the replicator dynamics. We first studied the uncontrolled fixed points of the replicator. We then showed that this class of problems admits a natural quasi-linearization, and that this quasi-linearized optimal control problem has a open-loop optimal control satisfying a specific second order differential equation. Furthermore, we showed that as the number of strategies grows, this differential equation admits a closed form solution. Numerical comparisons showed that this limiting case provides a natural approximation for the optimal control. We also showed that even when the starting conditions for the optimal control are far from the interior Nash equilibrium, where quasi-linearization is performed, we still can use it to approximate the optimal control in the original problem.

A promising direction for future research would be to extend these results to arbitrary circulant games or to cyclic games with a control parameter. In this context, a cyclic game is any circulant game in which the first row of the $\mathbf{L}_N$ matrix has form $(0,1,0,\dots,0,-1)$. Here the number of $0$'s between $1$ and $-1$ is $N-3$. The resulting matrix $\mathbf{M}_N$ with $1$'s corresponding to the $1'$ in $\mathbf{L}$ and $0$ elsewhere is the adjacency matrix of a cycle graph. Rock-paper-scissors is the only game that is both a cyclic and complete odd circulant graph (because the three-cycle is the complete graph on three vertices). Any circulant game whose payoff matrix can be written as $\mathbf{L}_N + \gamma\mathbf{M}_N$ will obey Equation \ref{eqn:dotgamma} because the circulant matrices form a commutative algebra. Beyond that, it is possible addition dynamics govern the optimal controls in these cases. This presents a logical area for further study.

Another extension of this work is to study circulant games with an off-center interior equilibrium point, rather than $\mathbf{u}^* = \tfrac{1}{N}\mathbf{1}$. Doing so, however, should introduce additional control parameters. This would be an interesting extension as well since this paper considered only a single control parameter. It also will make the application more realistic since an individual may have varying degrees of control over each population. In addition to introducing additional controls, another natural extension of this work is to derive controllers that drive the system toward a non-interior equilibrium point. In particular it would be intriguing to study the problem of deliberately eliminating one or more species.

Finally, studying more complex dynamics with control, like those found in the mutator-replicator, may produce interesting and useful results. However, these may not admit the necessary conditions to allow the control mechanisms identified in this paper to be applied.

\section*{Acknowledgement}
The authors thank Andrew Belmonte for his helpful comments and discussion. CG is supported by NSF grant CMMI-1463482 and CMMI-1932991.

\appendix

\section{Optimal Control Problems}\label{sec:OptimalControl}
In Section \ref{sec:GeneralControl} we introduce the problem of driving a population playing a cyclic game to its mixed strategy equilibrium. We present key facts from optimal control theory used in this study. Details are available in \cite{Mang66,Ki04,Friesz10}.

 A Bolza type optimal control problem is an optimization problem of the form:
 \begin{equation}\left\{
 \begin{aligned}
 \min\;\; & \Psi(\mathbf{x}(t_f)) + \int_{t_0}^{t_f} f(\mathbf{x}(t), \mathbf{u}(t),t) dt\\
 s.t.\;\; & \mathbf{\dot{x}} = \mathbf{g}(\mathbf{x}(t), \mathbf{u}(t),t),\\
 &\mathbf{x}(0) = \mathbf{x}_0.
 \end{aligned}\right.
 \label{eqn:Bolza}
 \end{equation}
 When $\Psi(\mathbf{x}(t_f)) \equiv 0$, this is called a Lagrange type optimal control problem. The vector of variables $\mathbf{x}$ is called the state, while the vector of decision variables $\mathbf{u}$ is called the control. Additional constraints on $\mathbf{u}$, $\mathbf{x}$ or the joint function of $\mathbf{x}$ and $\mathbf{u}$ can be added.

 The \textit{Hamiltonian} with adjoint variables (Lagrange multipliers) $\boldsymbol{\lambda}$ for this problem is:
 \begin{displaymath}
 \mathcal{H}(\mathbf{x},\boldsymbol{\lambda}, u) = f(\mathbf{x}(t), \mathbf{u}(t),t) + \boldsymbol\lambda^T\mathbf{g}(\mathbf{x}(t), \mathbf{u}(t),t).
 \end{displaymath}

 In what follows, we assume that all $f(\mathbf{x},\mathbf{u},t)$ and $\mathbf{g}(\mathbf{x},\mathbf{u},t)$ are continuous and differentiable in $\mathbf{x}$ and $\mathbf{u}$, and $\Psi(\mathbf{x}(t_f))$ is continuous and differentiable in $x(t_f)$. A proof of this lemma can be found in almost every book on optimal control (e.g. \cite{Friesz10}).

 \begin{lemma}[Necessary Conditions of Optimal Control] If $\mathbf{u}^*$ is a solution to Optimal Control Problem (\ref{eqn:Bolza}), then there is a vector of adjoint variables $\boldsymbol{\lambda}^*$ so that:
 \begin{equation}
 \mathcal{H}(\mathbf{x}^*(t), \mathbf{u}^*(t), \boldsymbol{\lambda}^*(t)) \leq \mathcal{H}(\mathbf{x}^*(t), \mathbf{u}(t), \boldsymbol{\lambda}^*(t))
 \end{equation}
 for all $t \in [0, T]$ and for all admissible inputs $\mathbf{u}$, and the following conditions hold:
 \begin{enumerate}
 \item Pontryagin's Minimim Principle: $\dot{\mathbf{u}}(t) = \frac{\partial \mathcal{H}}{\partial \mathbf{u}} = \mathbf{0}$ and $\frac{\partial^2\mathcal{H}}{\partial \mathbf{u}^2}$ is positive definite,
 \item Co-State Dynamics:
 \begin{displaymath}
 \dot{\boldsymbol{\lambda}}(t) = -\frac{\partial \mathcal{H}}{\partial \mathbf{x}} = -\boldsymbol{\lambda}^T(t)\frac{\partial \mathbf{g}(\mathbf{x}, \mathbf{u})}{\partial \mathbf{x}} + \frac{\partial f(\mathbf{x}, \mathbf{u})}{\partial \mathbf{x}},
 \end{displaymath}
 \item State Dynamics: $\dot{x}(t) = \frac{\partial \mathcal{H}}{\partial \boldsymbol{\lambda}} = \mathbf{g}(\mathbf{x}, \mathbf{u})$,
 \item Initial Condition: $\mathbf{x}(0) = \mathbf{x}_0$, and 
 \item Transversality Condition: $\boldsymbol{\lambda}(t_f) = \frac{\partial \Psi}{\partial \mathbf{x}}(\mathbf{x}(t_f))$.
 \end{enumerate}
 \label{lem:OptConNec}
 \end{lemma}

 We will use the following restricted form of Mangasarian's Sufficiency condition \cite{Mang66} to argue a controller we derive in Section \ref{sec:CyclicControl} is the optimal controller.

 \begin{lemma}[Mangasarian's Sufficiency Condition - Restricted Form] Suppose $(\mathbf{x}^*,\mathbf{u}^*)$ satisfies the necessary conditions from Lemma \ref{lem:OptConNec} and  $\mathcal{H}$ is jointly convex in $\mathbf{x}$ and $\mathbf{u}$ for all time, and $\Psi(\mathbf{x}(t_f)) \equiv 0$. Then $(\mathbf{x}^*,\mathbf{u}^*)$ is a globally optimal control in the sense that it minimizes the objective functional.
 \label{lem:OptConSuff}
 \end{lemma}

 We note that Mangasarian's Sufficiency Condition specifically implies the \textit{strong Legendre-Clebsch necessary conditions} for optimality of the control:
 \begin{enumerate}
 \item $\mathcal{H}_u = 0$, and
 \item $\mathcal{H}_{uu} < 0$.
 \end{enumerate}
 In the paper, we also discuss the sufficiency of the boundedness of the matrix Riccati equation. Since this plays only a small role in our overall analysis, we introduce this when it is needed.

\section{Technical Proofs}\label{sec:TechProofs}

\begin{proof}[Proof of Lemma \ref{lem:DF}] We prove the result for Row 1 of $D_\mathbf{u}\mathbf{F}$. The remainder of the argument follows from the circulant structure of $\mathbf{L}_N$. We have:
\begin{displaymath}
\mathbf{F}_1(\mathbf{u}) = u_1\mathbf{e}_1^T\mathbf{L}_N\mathbf{u},
\end{displaymath}
because $\mathbf{u}^T\mathbf{L}_N\mathbf{u} = 0$. Note:
\begin{equation}
u_1\mathbf{e}_1^T\mathbf{L}_N\mathbf{u} = u_1\left(\sum_{j=2}^{N}(-1)^{j-1}u_j\right).
\label{eqn:Row1}
\end{equation}
Differentiating with respect to $u_1$ and evaluating at $u_1 = u_2 = \cdots = u_n = \tfrac{1}{N}$ yields:
\begin{displaymath}
\left[D_\mathbf{u}\mathbf{F}\right]_{1,1} = \frac{1}{N}\sum_{j=2}^{N}(-1)^{j-1} = 0,
\end{displaymath}
since $N$ is odd. Differentiating Expression \ref{eqn:Row1} with respect to $u_j$ and evaluating at $u_1 = u_2 = \cdots = u_n = \tfrac{1}{N}$ yields:
\begin{displaymath}
\left[D_\mathbf{u}\mathbf{F}\right]_{1,j} = \frac{(-1)^{j-1}}{N} = \frac{1}{N}\mathbf{L}_{N_{1,j}}.
\end{displaymath}
The result now follows from the fact that $\mathbf{L}_N$ is a circulant matrix.
\end{proof}

\begin{proof}[Proof of Lemma \ref{lem:DG}] We show the result for Row 1 of $D_\mathbf{u}\mathbf{G}$. The remainder of the argument follows from the circulant structure of $\mathbf{M}_N$. We have already noted in Corollary \ref{cor:uu} that:
\begin{displaymath}
\mathbf{u}^T\mathbf{M}_N\mathbf{u} = \sum_{i=1}^N\sum_{j > i} u_iu_j.
\end{displaymath}
We compute:
\begin{displaymath}
\mathbf{e}_1^T\mathbf{M}_N\mathbf{u} = \sum_{j=1}^{(N-1)/2} u_{2j+1}.
\end{displaymath}
Differentiate with respect to $k = 2j+1$ for $j \in \{1,\dots,n\}$, corresponding to a non-zero index in the first row of $\mathbf{M}_N$. We have:
\begin{displaymath}
\frac{\partial\mathbf{G}}{\partial u_k} = u_1\left(1 - \left(\sum_{j \neq k} u_j\right)\right).
\end{displaymath}
Evaluating at $u_1 = u_2 = \cdots = u_n = \tfrac{1}{N}$ we obtain:
\begin{displaymath}
\left[D_\mathbf{u}\mathbf{G}\right]_{1,k} = \frac{1}{N}\left(1 - \frac{N-1}{N}\right) = \frac{1}{N^2},
\end{displaymath}
for $k=2j+1$ with $j \in \{1,\dots,n\}$. Differentiate now with respect to $u_1$ to obtain:
\begin{equation}
\frac{\partial\mathbf{G}}{\partial u_1} = \sum_{j=1}^{(N-1)/2} u_{2j+1} - 2\sum_{j = 2}^N u_1u_j - \sum_{i=2}^N\sum_{j > i}u_{i}u_{j}.
\end{equation}
Evaluating at $u_1 = u_2 = \cdots = u_n = \tfrac{1}{N}$ we obtain:
\begin{displaymath}
\left[D_\mathbf{u}\mathbf{G}\right]_{1,1} =
\frac{n}{N} - 2\frac{2n}{N^2} - \frac{1}{N^2}\left(\frac{1}{2}(N-2)(N-1)\right) = 
\frac{-N^2-2nN+N-2}{N^2} = -\frac{2n}{N},
\end{displaymath}
when $2n+1$ is substituted for $N$ in the numerator. Finally, consider $k \neq 1$ and $k \neq 2j+1$ for $j \in \{1,\dots,n\}$. Differentiating with respect to $u_k$ we have:
\begin{displaymath}
\frac{\partial\mathbf{G}}{\partial u_k} = -u_1\left(\sum_{j \neq k}u_j\right).
\end{displaymath}
Evaluating at $u_1 = u_2 = \cdots = u_n = \tfrac{1}{N}$ we obtain:
\begin{displaymath}
\left[D_\mathbf{u}\mathbf{G}\right]_{1,k} = -\frac{1}{N^2}(N-1) = -\frac{2n}{N}.
\end{displaymath}
The result now follows from the fact that $\mathbf{M}_N$ is a circulant matrix.
\end{proof}

\begin{proof}[Proof of Proposition \ref{prop:DotGammaSign}] Consider any vector $\mathbf{x} \in \mathbb{R}^N$. Then:
\begin{multline*}
\mathbf{x}^T\mathbf{H}\mathbf{x} = \frac{1}{N^2}\mathbf{x}^T\left(N\left(\mathbf{M}_N-\mathbf{1}_N\right) + \mathbf{1}_N\right)\mathbf{x}^T = \\
\frac{1}{N^2}\left(N\left(\sum_{i=1}^N\sum_{j>i}x_i x_j - \sum_{i=1}^N x_i^2 -2\sum_{i=1}^N\sum_{j > i}x_ix_j \right) +
\sum_{i=1}^N x_i^2 + 2\sum_{i=1}^N\sum_{j > i}x_ix_j\right) = \\
-\frac{N-1}{N^2}\sum_{i=1}^N x_i^2 - \frac{N-2}{N^2} \sum_{i=1}^N\sum_{j > i}x_ix_j.
\end{multline*}
Let $\mathbf{S} \in \mathbb{R}^{N \times N}$ be the upper-triangular matrix defined as:
\begin{displaymath}
\mathbf{S}_{ij} = \begin{cases}
-\frac{N-1}{N^2} & \text{if $i = j$},\\
-\frac{N-2}{N^2} & \text{otherwise}.
\end{cases}
\end{displaymath}
Then:
\begin{displaymath}
\mathbf{x}^T\mathbf{H}\mathbf{x} = \mathbf{x}^T\mathbf{S}\mathbf{x} =
-\frac{N-1}{N^2}\sum_{i=1}^N x_i^2 - \frac{N-2}{N^2} \sum_{i=1}^N\sum_{j > i}x_ix_j.
\end{displaymath}
The leading principal minors of $\mathbf{S}$ alternate in sign (the diagonal is entirely negative) and thus by Sylvester's criterion, $\mathbf{S}$ is negative definite. It follows at once that $\mathbf{x}^T\mathbf{H}\mathbf{x} < 0$ for all $\mathbf{x} \neq \mathbf{0}$ and thus $\mathbf{H}$ is negative definite. The fact that $\dot{\gamma} \leq 0$ now follows from Lemma \ref{lem:MainControl1}.
\end{proof}

\begin{proof}[Proof of Lemma \ref{lem:5H}] From Lemma \ref{lem:DG} we have:
\begin{displaymath}
\mathbf{H} = \frac{1}{N^2}\left(N\left(\mathbf{M}_N - \mathbf{1}_N\right) + \mathbf{1}_N\right).
\end{displaymath}
Let $\mathbf{R} = \mathbf{M}_N - \mathbf{1}_N$. The following computations are straight forward:
\begin{align*}
\mathbf{H}^T\mathbf{H} &= \frac{1}{N^4}\left(N^2\mathbf{R}^T\mathbf{R} + N\mathbf{R}^T\mathbf{1}_N + N\mathbf{1}^T_N\mathbf{R} + \mathbf{1}_N^T\mathbf{1}_N \right),\\
\mathbf{H}\mathbf{H} &= \frac{1}{N^4}\left(N^2\mathbf{R}\mathbf{R} + N\mathbf{R}\mathbf{1}_N + N\mathbf{1}_N\mathbf{R} + \mathbf{1}_N\mathbf{1}_N \right).
\end{align*}
Note that:
\begin{displaymath}
\mathbf{1}^T_N\mathbf{1}_N = \mathbf{1}_N\mathbf{1}_N = N\mathbf{1}_N.
\end{displaymath}
We may also compute:
\begin{displaymath}
\mathbf{1}_N\mathbf{M}_N = \mathbf{1}^T_N\mathbf{M}_N = n\mathbf{1}_N,
\end{displaymath}
because $M_N$ contains $n$ unit entries in each column (row). Consequently:
\begin{displaymath}
\left(\mathbf{M}_N^T\mathbf{1}_N\right)^T = \mathbf{1}^T\mathbf{M}_N = n\mathbf{1}_N,
\end{displaymath}
and therefore:
\begin{displaymath}
\mathbf{M}_N^T\mathbf{1}_N = n\mathbf{1}_N^T=n\mathbf{1}_N.
\end{displaymath}
Using this information, we compute:
\begin{equation}
\mathbf{R}^T\mathbf{1}_N = \mathbf{1}_N^T\mathbf{R} = \mathbf{1}_N\mathbf{R} = \mathbf{R}\mathbf{1}_N = -(n+1)\mathbf{1}_N.
\label{eqn:R1}
\end{equation}
Thus:
\begin{multline*}
\mathbf{H}^T\mathbf{H} + \mathbf{H}\mathbf{H} =
\frac{1}{N^4}\left(
N^2\left(\mathbf{R}^T + \mathbf{R}\right)\mathbf{R} -4N(n+1)\mathbf{1}_N + 2N\mathbf{1}_N\right) = \\
\frac{1}{N^4}\left(N^2\left(\mathbf{R}^T + \mathbf{R}\right)\mathbf{R} - 2N\left(2(n+1)-1\right)\mathbf{1}_N
\right) = \\
\frac{1}{N^2}\left(N^2\left(\mathbf{R}^T + \mathbf{R}\right)\mathbf{R} -2N^2\mathbf{1}_N\right) =
\frac{1}{N^2}\left(\left(\mathbf{R}^T + \mathbf{R}\right)\mathbf{R} - 2\mathbf{1}_N\right).
\end{multline*}
Using the fact that $\mathbf{H} = (N\mathbf{R} + \mathbf{1}_N)/N^2$, we may write:
\begin{displaymath}
\mathbf{H}^T\mathbf{H} + \mathbf{H}\mathbf{H} + \mathbf{H} =
\frac{1}{N^2}\left( \left(\mathbf{R}^T + \mathbf{R} + N\mathbf{I}_N\right)\mathbf{R} - \mathbf{1}_N \right).
\end{displaymath}
The circulant structure of $\mathbf{M}$ implies the identity:
\begin{displaymath}
\mathbf{M}^T + \mathbf{M} = \mathbf{1}_N - \mathbf{I}_N.
\end{displaymath}
Therefore:
\begin{displaymath}
\mathbf{R}^T + \mathbf{R} = \mathbf{1}_N - \mathbf{I}_N - 2\mathbf{1}_N = -\mathbf{1}_N - \mathbf{I}_N.
\end{displaymath}
Thus, using Equation \ref{eqn:R1} and the fact that $N - 1 = 2n$:
\begin{multline*}
\mathbf{H}^T\mathbf{H} + \mathbf{H}\mathbf{H} + \mathbf{H} =
\frac{1}{N^2}\left(\left((N-1)\mathbf{I}_N - \mathbf{1}_N \right)\mathbf{R} - \mathbf{1}_N \right) = \\
\frac{1}{N^2}\left((N-1)(\mathbf{M}_N - \mathbf{1}_N) + (n+1)\mathbf{1}_N - \mathbf{1}_N \right) = \\
\frac{1}{N^2}\left(2n\mathbf{M}_N + -n\mathbf{1}_N\right) =
\frac{n}{N^2}\left(2\mathbf{M}_N - \mathbf{1}_N\right).
\end{multline*}
Recall from Corollary \ref{cor:uu} that:
\begin{displaymath}
\mathbf{x}^T\mathbf{M}_N\mathbf{x} = \sum_{i=1}^N\sum_{j > i}x_ix_j.
\end{displaymath}
Furthermore, it is straight forward to compute:
\begin{displaymath}
\mathbf{x}^T\mathbf{1}_N\mathbf{x} = \sum_{i=1}^N x_i^2 + 2\sum_{i=1}^N\sum_{j > i}x_ix_j.
\end{displaymath}
Therefore:
\begin{multline*}
\mathbf{x}^T\left( \mathbf{H}^T\mathbf{H} + \mathbf{H}\mathbf{H} + \mathbf{H} \right)\mathbf{x} =
\frac{n}{N^2}\mathbf{x}^T\left(2\mathbf{M}_N - \mathbf{1}_N\right)\mathbf{x} = \\
\frac{n}{N^2}\left(2\sum_{i=1}^N\sum_{j > i}x_ix_j - \sum_{i=1}^N x_i^2 -
2\sum_{i=1}^N\sum_{j > i}x_ix_j\right) = -\frac{n}{N^2}\left\lVert\mathbf{x}\right\rVert^2.
\end{multline*}
This completes the proof.
\end{proof}

\section{Examples with $N = 3, 5, 7,9$}\label{app:Examples}
We study the derived optimal controls for the case when $N = 3, 5, 7, 9$ in both the fully non-linear optimal control problem and the quasi-linearized optimal control problem. In particular we observe similar structure to the optimal controls in all cases. For these examples, we set $t_f = 6$, except in the last case where we extend it to show an example where the sufficient condition for optimality is not satisfied.

In Figure \ref{fig:3cyclic} we show the optimal control for both the non-linear and quasi-linearized optimal control problems. We set $r = 0.2$ and the initial state $\mathbf{u}_0 = (0.2333, 0.3333, 0.43333)$, which is close enough to the fixed point for the simpler quasi-linearized approximation to be valid. We also demonstrate the equivalence between the solution to the quasi-linearized Euler-Lagrange equations and the second order differential equation (Eq. \ref{eqn:2ODE-Full}). We also show the approximation to the quasi-linearized control function that arises as a solution to Equation \ref{eqn:2ODE}.
\begin{figure}[htbp]
\centering
\includegraphics[scale=0.24]{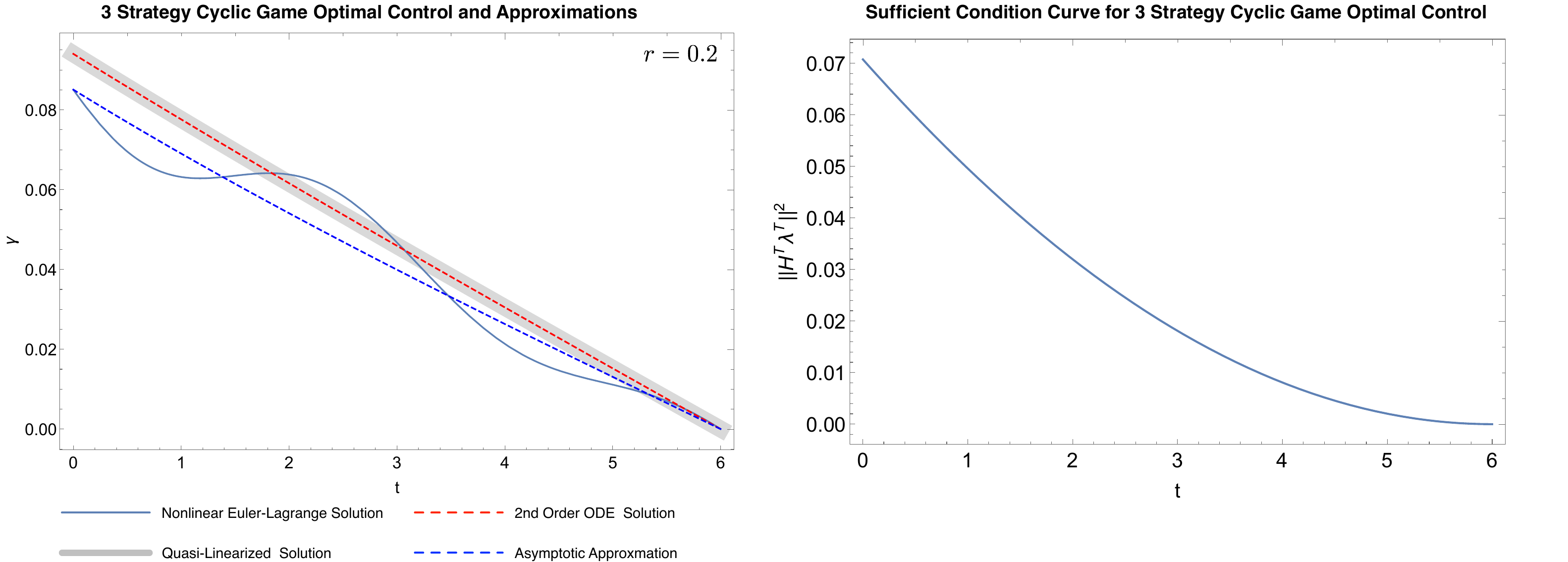}
\caption{The control function, its approximations and the $\norm{\mathbf{H}^T\bm{\lambda}^T}^2$, used in determining whether the solution to the necessary conditions are sufficient for an optimal control for the 3 strategy cyclic game (rock-paper-scissors).}
\label{fig:3cyclic}
\end{figure}
In the 3 strategy case, the condition for optimality is:
\begin{equation}
r > \norm{\mathbf{H}^T\bm{\lambda}^T}^2 = \frac{1}{9}\norm{\bm{\lambda}}^2.
\end{equation}
This is not generally true, but it is equivalent to the sufficient condition derived in \cite{GB17} for the diffeomorphic triple public goods game. Thus, Theorem \ref{thm:Sufficient} generalizes the results from \cite{GB17}.

\begin{figure}[htbp]
\centering
\includegraphics[scale=0.24]{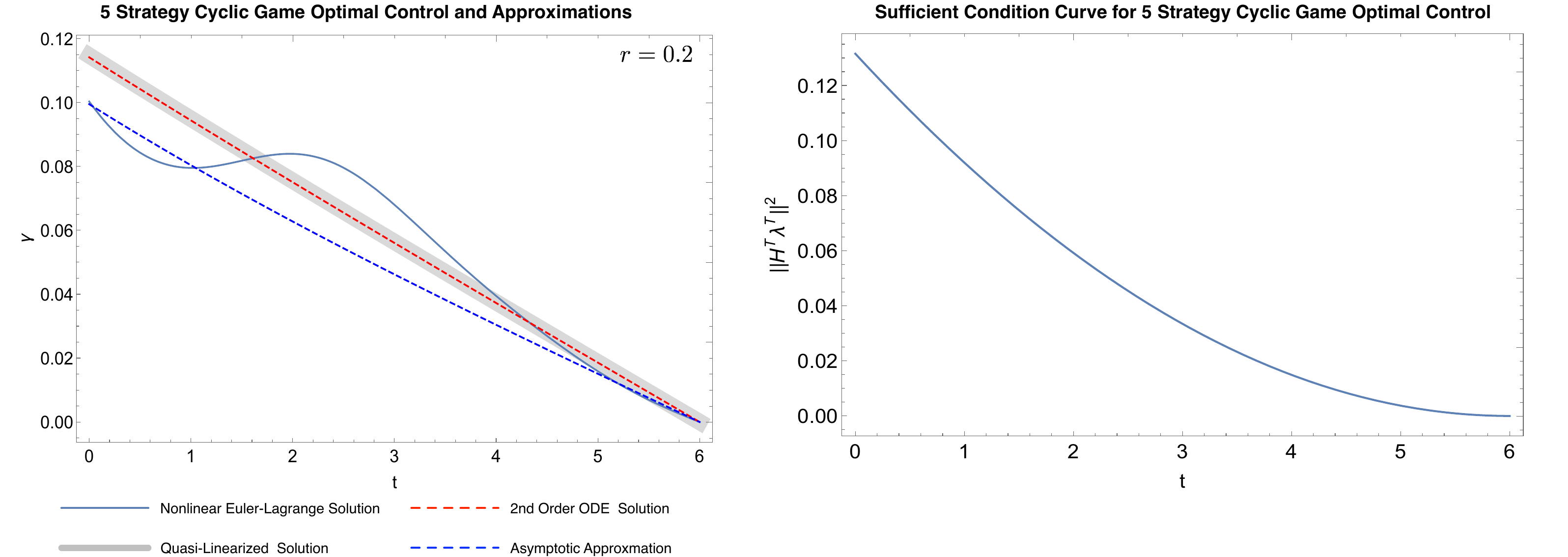}
\caption{The control function, its approximations and the $\norm{\mathbf{H}^T\bm{\lambda}^T}^2$, used in determining whether the solution to the necessary conditions are sufficient for an optimal control for the 5 strategy cyclic game (rock-paper-scissors-Spock-lizard).}
\label{fig:5cyclic}
\end{figure}

In Figure \ref{fig:5cyclic} we show the relevant control plots for the 5 strategy cyclic game (rock-paper-scissors-Spock-lizard\footnote{Developed by Sam Kass. See \url{http://www.samkass.com/theories/RPSSL.html} or Episode 8, Season 2 of \textit{The Big Bang Theory}. Note, to properly organize the moves to produce a  circulant matrix, the strategies should be ordered as rock-paper-scissors-Spock-lizard rather than rock-paper-scissors-lizard-Spock as they are on \textit{The Big Bang Theory}. Kass correctly organizes the strategies.}). We again use $r = 0.2$ and set $\mathbf{u}_0 = (0.1,0.3,0,0.1,0.3).$

An interesting feature of the 5 strategy cyclic game is that the limiting approximation (Equation \ref{eqn:2ODE}) does not perform as well as it did for the 3 strategy game. As we see in Figures \ref{fig:7cyclic} and \ref{fig:9cyclic}, the approximation does improve (as we expect) as $N$ increases. This anomalous behavior may be a function of numerical instability or a property of the 5 strategy cyclic game. We do note that because of the properties of Equations \ref{eqn:2ODE-Full} and \ref{eqn:2ODE} (i.e, branching solutions), we did observe some numerical instability when simulating these systems. We note that in the 5 strategy case, the sufficient condition on optimality is satisfied.

In Figures \ref{fig:7cyclic} and \ref{fig:9cyclic} we illustrate the optimal control for 7 and 9 strategy games. In these cases $r = 0.2$ again. To maintain feasibility of the starting solution, we set $\mathbf{u}_0$ by alternately adding and subtracting $0.05$ from the equilibrium, but kept the third strategy at proportion $1/N$ in both cases. Thus we assured $\mathbf{u}_0$ was in $\Delta_N$ in both cases.
\begin{figure}[htbp]
\centering
\includegraphics[scale=0.23]{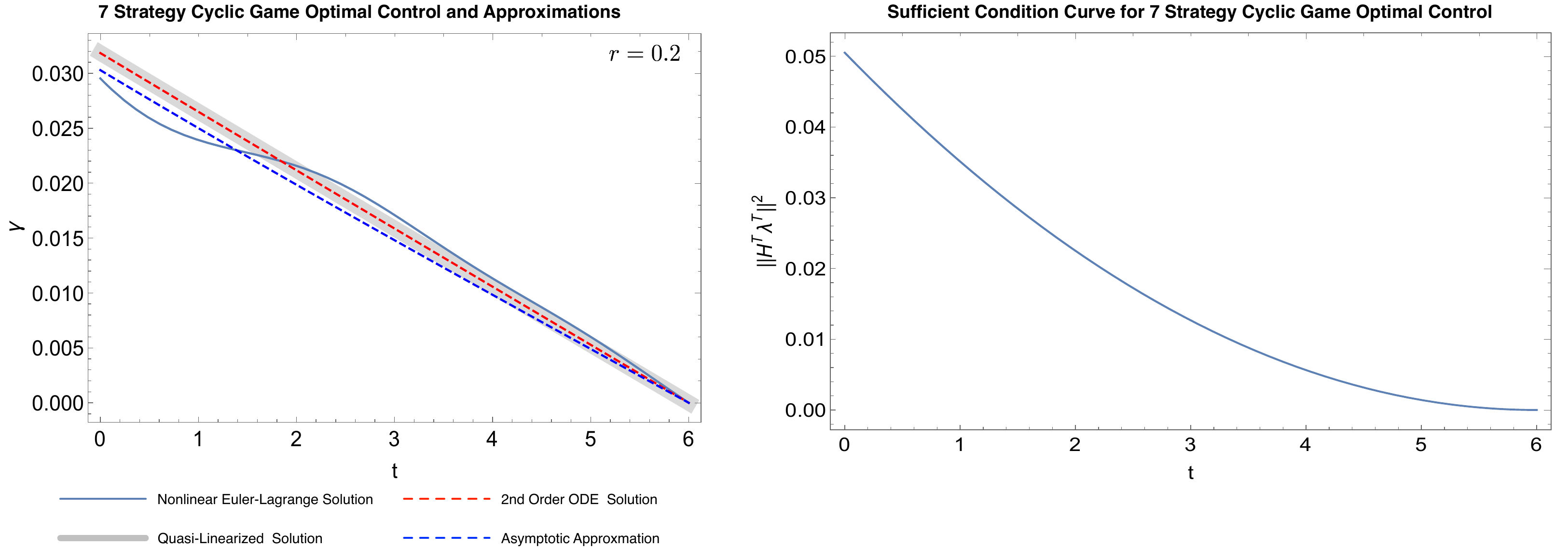}
\caption{The control function, its approximations and the $\norm{\mathbf{H}^T\bm{\lambda}^T}^2$, used in determining whether the solution to the necessary conditions are sufficient for an optimal control for the 7 strategy cyclic game.}
\label{fig:7cyclic}
\end{figure}
\begin{figure}[htbp]
\centering
\includegraphics[scale=0.23]{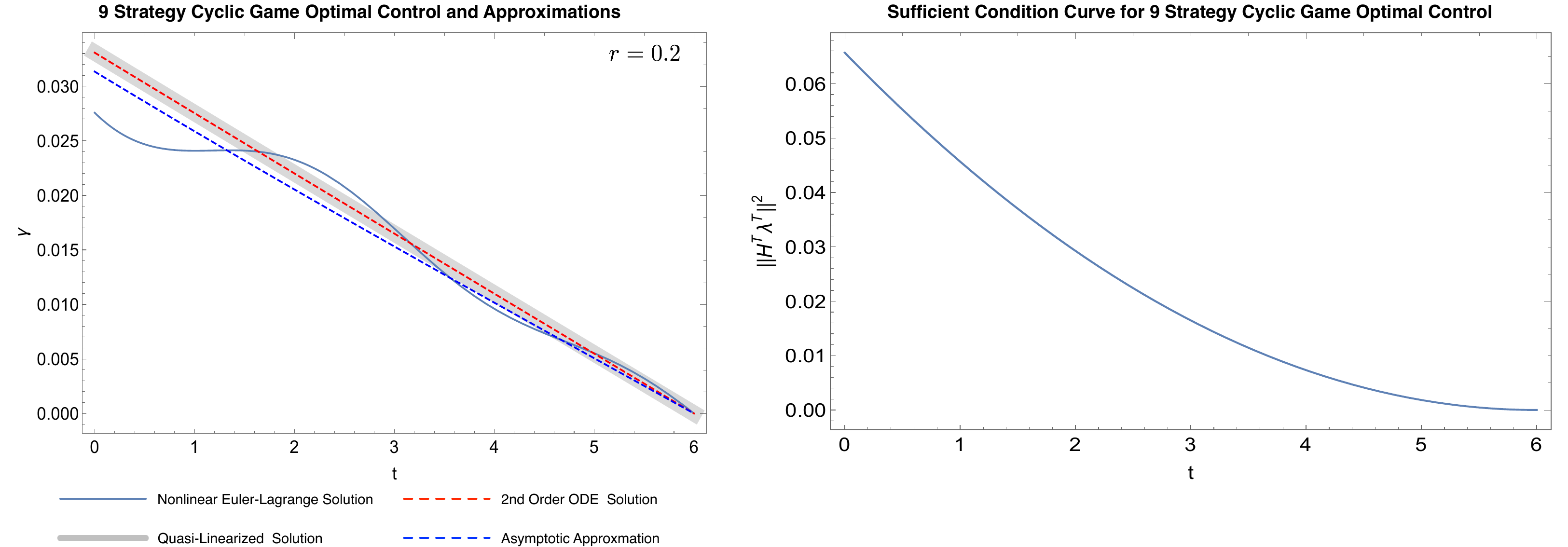}
\caption{The control function, its approximations and the $\norm{\mathbf{H}^T\bm{\lambda}^T}^2$, used in determining whether the solution to the necessary conditions are sufficient for an optimal control for the 9 strategy cyclic game.}
\label{fig:9cyclic}
\end{figure}
As we expect, as $N$ increases, the approximation in Equation \ref{eqn:2ODE} improves. It is also interesting to note that the general structure of the optimal control function is similar in all cases with the quasi-linearized control. It exhibits almost linear behavior, and the fully non-linear controller shows decreasing oscillation. That the optimal controller is a decreasing function is consistent with Proposition \ref{prop:DotGammaSign}.

We can analyze the control problem even when the starting state is not \textit{near} the fixed point, which yields a case where the sufficient condition for optimality fails to hold. We again consider the case when $N = 3$ and extend the time horizon of control to $t_f = 15$. We start at the point $\mathbf{u}_0 = (0.8,0.1,0.1)$, which is not near the equilibrium point $\mathbf{u}^* = \tfrac{1}{3}\mathbf{1}$, thus reducing the accuracy of the quasi-linearized approximation. The objective of this example is to study both the extended time control horizon as well as the control that results when the starting point is further from the equilibrium point.

We compute an optimal control using the fully non-linear Euler-Lagrange equations, the Euler-Lagrange equations for the quasi-linearized system, and the exact differential equation for $\gamma$ given in Lemma \ref{lem:MainControl1}. As demonstrated in Figure \ref{fig:Riccati}, the quasi-linearized control functions are identical to each other (as expected), and they are highly correlated to the control function for the fully non-linear system. The co-state, however does not satisfy the sufficient condition $r > \mathbf{H}^T\bm{\lambda}$. Here $r = 0.2$, as in the previous examples.
\begin{figure}[htbp]
\centering
\subfigure[Control and Co-State Plot]{\includegraphics[scale=0.24]{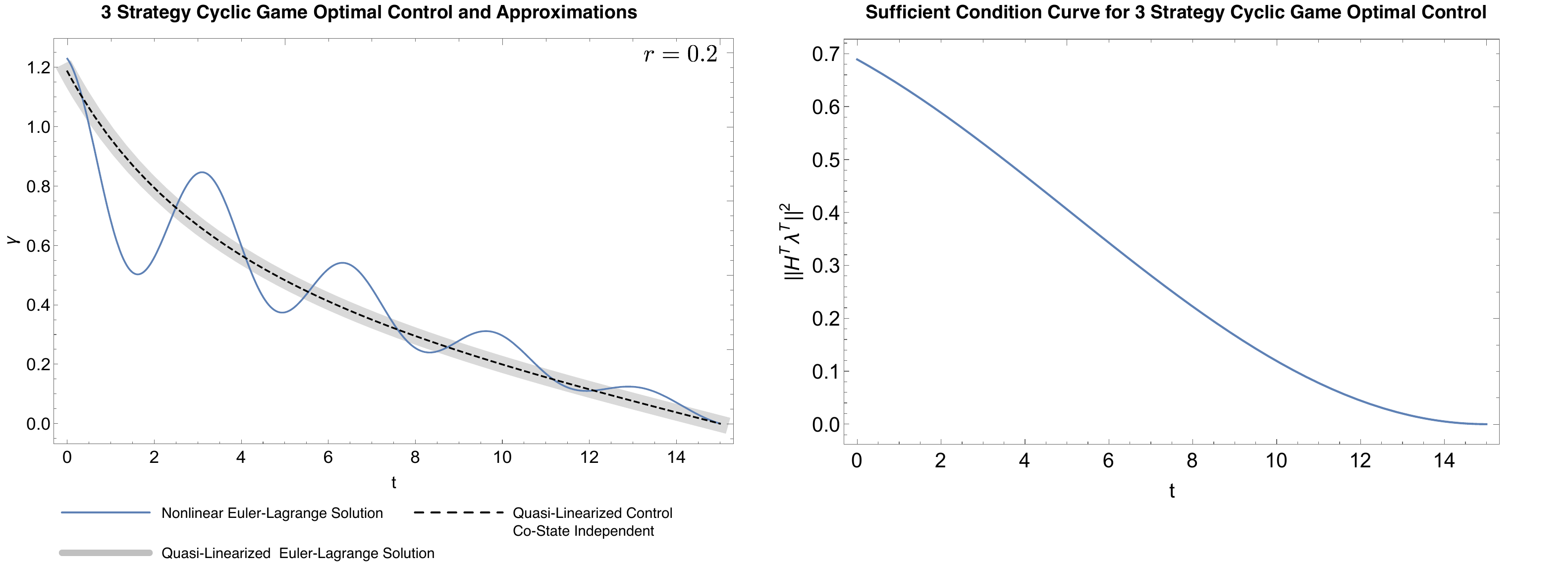}}
\subfigure[Matrix Riccati Equation Solution Curves]{\includegraphics[scale=0.35]{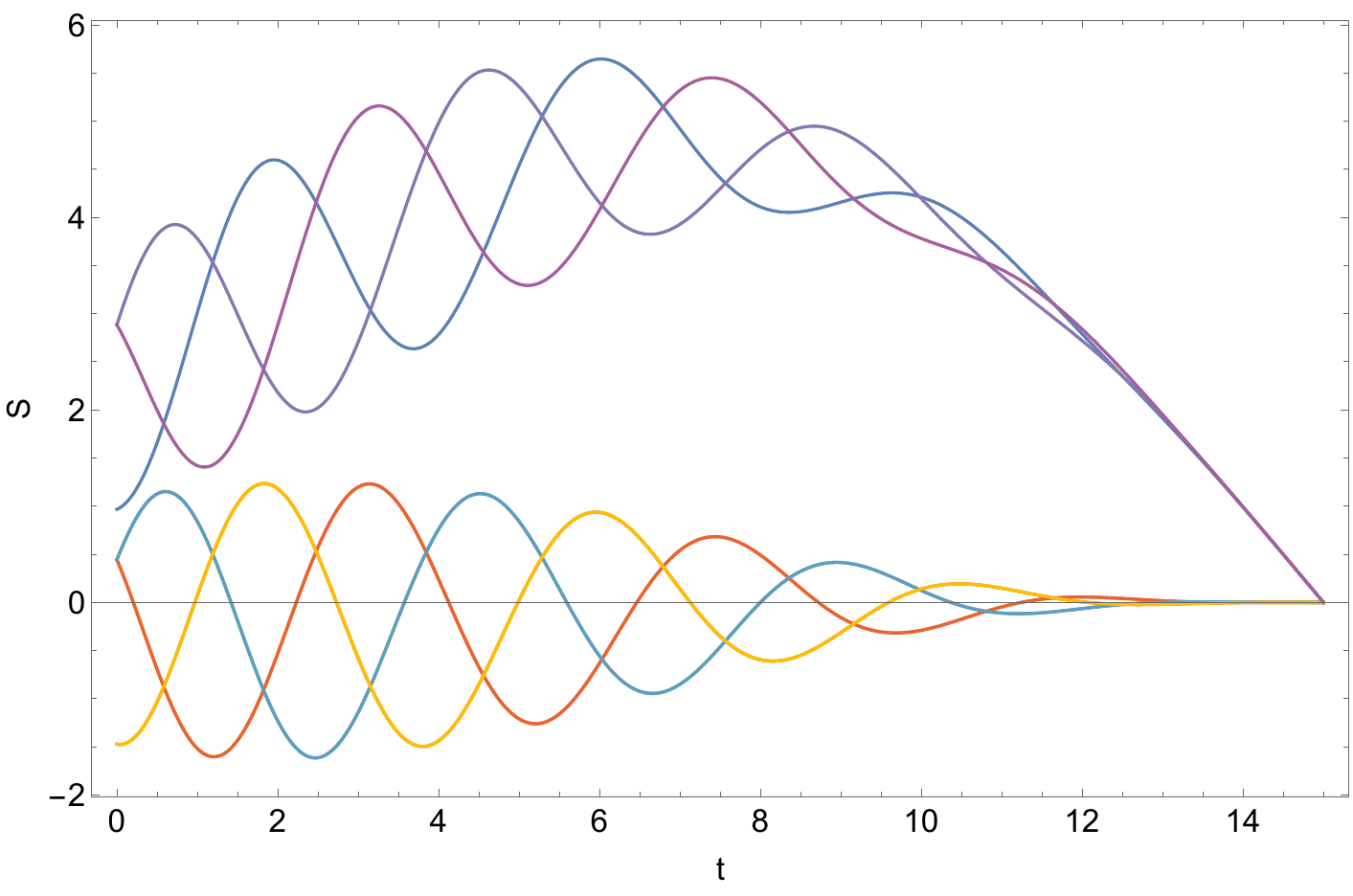}}
\caption{A solution that violates the sufficient condition for optimality $r > \mathbf{H}^T\bm{\lambda}$ but can be shown to be optimal by appealing to the matrix Riccatti equation, which has bounded solutions for $t \in [0,t_f]$. We compare the solution to the ordinary controller derived from the ordinary Euler-Lagrange equations and the controller that is directly computed using Lemma \ref{lem:MainControl1}. Note they are identical as expected.}
\label{fig:Riccati}
\end{figure}
We can still analyze this problem by solving the matrix Riccati equation as given in Theorem \ref{thm:main} to show that it is bounded, and thus the control identified for the quasi-linearized system is optimal. The matrix Riccati equation for this system is:
\begin{align*}
-\dot{\mathbf{S}} &= \mathbf{I}_3 + \left(\mathbf{J}+\gamma\mathbf{H}\right)^T\mathbf{S} +
\mathbf{S}\left(\mathbf{J} + \gamma\mathbf{H}\right) - \\
& \hspace*{5em}\frac{1}{r}\left(\bm{\lambda}^T\mathbf{H}+\mathbf{x}^T\mathbf{H}^T\mathbf{S}\right)^T\left(\bm{\lambda}^T\mathbf{H}+\mathbf{x}^T\mathbf{H}^T\mathbf{S}\right),\\
\mathbf{S}(t_f) &= \mathbf{0}.
\end{align*}
This system of equations contains nine equations because $\mathbf{S}$ is $3 \times 3$. As shown in Figure \ref{fig:Riccati}, the solution curves for this equation are all bounded on the time interval of consideration. Furthermore, since the state and co-state necessarily satisfy the Euler-Lagrange equations, and the Hamiltonian is convex in $\gamma$ and $\gamma$ (and therefore maximizes the Hamiltonian at all times (see Lemma \ref{lem:NecessaryControl})), the control function must be (locally) optimal.

Note that while the starting point in this example is not near the equilibrium point, the control computed for the quasi-linear approximation is still highly correlated to the control computed for the non-linear system. Thus, in this case quasi-linearization still provides a reasonable approximation to the optimal control in the non-linear system.

\bibliographystyle{unsrt}
\bibliography{mrabbrev,RPSControlBib,RPS-General}

\end{document}